\documentclass[11pt, leqno]{amsart}
\usepackage{amsmath,amsthm,amsfonts,amssymb,amscd,tikz,amscd}

\newcommand{\bbA}{{\mathbb A}}
\newcommand{\bbC}{{\mathbb C}}

\newcommand{\bbL}{{\mathbb L}}
\newcommand{\bbR}{{\mathbb R}}

\newcommand{\bbZ}{{\mathbb Z}}

\newcommand{\bbQ}{{\mathbb Q}}

\newtheorem{thm}{Theorem}[section]

\newtheorem{theorem}[thm]{Theorem}

\newtheorem{corollary}[thm]{Corollary}
\newtheorem{prop}[thm]{Proposition}
\newtheorem{lemma}[thm]{Lemma}

\newtheorem{remark}[thm]{Remark}

\newtheorem{defn}[thm]{Definition}

\newtheorem*{claim*}{Claim}
\numberwithin{equation}{section}

\begin{document}

\begin{abstract} Let $X$ be a smooth toric variety defined by the fan
$\Sigma$. We consider $\Sigma$ as a finite set with topology and define a natural sheaf of graded algebras $\mathcal{A}_\Sigma$ on $\Sigma$. The category of modules over $\mathcal{A}_\Sigma$ is studied (together with other related categories). This leads to a certain combinatorial Koszul duality equivalence.

We describe the equivariant category of coherent sheaves $coh_{X,T}$ and a related (slightly bigger) equivariant category $\mathcal{O}_{X,T}\text{-mod}$ in terms of sheaves of modules over the sheaf of algebras $\mathcal{A}_\Sigma$. Eventually (for a complete $X$) the combinatorial Koszul duality is interpreted in terms of the Serre functor on $D^b(coh_{X,T})$.
\end{abstract}

\title{Derived category of equivariant coherent sheaves on a smooth toric variety and Koszul duality}

\author{Valery A. Lunts}
\address{Department of Mathematics,
Indiana University,
Bloomington, IN, USA
47405,}

\address{National Research University Higher School of Economics, Moscow, Russia}

\thanks{The  author was supported by the Basic Research Program of the National Research University Higher School of Economics}

\maketitle

\section{Introduction}

Let $T$ be a complex torus, $M=Hom (T,\bbC ^*)$ its group of characters and $N=Hom(\bbC ^*,T)$ its dual group of 1-parameter subgroups. A normal toric variety $X$ can be described by a rational polyhedral fan $\Sigma $ in the vector space $N_{\bbR}$. One has several interesting (derived) categories of sheaves on the toric variety $X$ and it is tempting to try to describe those in terms of simple combinatorial structures on the fan $\Sigma$.
Several natural examples of such categories include:

(1) The equivariant derived category $D^b_{T,c}(X)$ of constructible sheaves on $X$ \cite{BeLu}.

(2) The (nonequivariant) derived category $D_{loc}^b(X)$ consisting of complexes of sheaves on $X$, which are locally constant (possibly of infinite rank) on each $T$-orbit.

(3) The equivariant derived category $D^b(coh_{X,T})$ of coherent sheaves on $X$.

Note that in the first two examples we consider the space $X$ with the analytic topology, and in the third one - with the Zariski topology.

The category in (1) was described in \cite{Lu} in the following way:
we consider $\Sigma$ as a topological space (with finitely many points) and the natural topology given by the identification $\Sigma =X/T$. It has a natural sheaf $\mathcal{C}_\Sigma$ of graded algebras, where the stalk $\mathcal{C}_{\Sigma ,\sigma}$ at the cone $\sigma \in \Sigma$ is the (evenly graded) algebra of polynomial functions on the linear span $\langle \sigma\rangle$ of $\sigma$. (This graded algebra $\mathcal{C}_{\Sigma ,\sigma}$ is canonically isomorphic to the cohomology ring of the classifying space of the subtorus $T_\sigma \subset T$ with the Lie algebra $\langle \sigma\rangle$). There is a natural equivalence of triangulated categories
\begin{equation} D^b_{T,c}(X)\simeq D^b_c(\mathcal{C}_\Sigma \text{-dgmod})
\end{equation}
where $D^b_c(\mathcal{C}_\Sigma \text{-dgmod})$ is a natural subcategory of the derived category of dg-modules over the sheaf of algebras $\mathcal{C}_\Sigma$.

The category (2) (or rather its version $D_{loc,unip}^b(X)$, when only local systems with pro-unipotent monodromy are allowed) was described in \cite{BrLu} in terms of {\it cosheaves} over the sheaf of algebras $\mathcal{C}_\Sigma$. These combinatorial descriptions of the categories (1) and (2) led to a Koszul duality
between certain closely related categories for dual affine toric varieties (see \cite{BrLu} for details).

In this paper we are concerned with the category (3) in case $X$ is smooth. The abelian category $coh_{X,T}$ is contained in a slightly bigger  category which we denote by $\mathcal{O}_{X,T}\text{-mod}$. It is the full abelian subcategory of $T$-equivariant sheaves of $\mathcal{O}_X$-modules, generated by extensions by zero $\mathcal{O}_{st({\bf o})}(m)$ from stars of orbits of the twisted structure sheaf (Definition \ref{defn of good subcat}). This category
$\mathcal{O}_{X,T}\text{-mod}$ is decsribed as follows: Define a sheaf of $M$-graded algebras $\mathcal{A}_{\Sigma}$ on $\Sigma$, with stalks
$$\mathcal{A}_{\Sigma ,\sigma}=A(st({\bf o}_\sigma))$$
where $A(st({\bf o}_\sigma))$ is the ring of regular functions on the star of the orbit corresponding to the cone $\sigma$. The category
$\mathcal{O}_{X,T}\text{-mod}$ is then equivalent to the category of finitely generated $\mathcal{A}_{\Sigma}$-modules (Proposition \ref{prop on comb equiv}):
\begin{equation}\label{intro first} \mathcal{O}_{X,T}\text{-mod}
\simeq \mathcal{A}_{\Sigma}\text{-mod}
\end{equation}
This equivalence induces the full and faithful functor (Theorem \ref{summary on der equiv and emb}):
$$\psi: D^b(coh_{X,T})\hookrightarrow D^b(\mathcal{A}_{\Sigma}\text{-mod})$$

Similarly one defines the abelian category $\text{co-}\mathcal{A}_{\Sigma}\text{-tcf}$ of {\it torsion cofinite} cosheaves of $\mathcal{A}_{\Sigma}$-modules. There exists a Koszul duality functor (Theorem \ref{koz dual thm}):
$${\bf K}:D^b(\mathcal{A}_{\Sigma}\text{-mod})\to D^b(\text{co-}\mathcal{A}_{\Sigma}\text{-tcf})$$
which is essentially taken from \cite{BrLu}. Finally we construct a natural functor
$$\phi :D^b(\text{co-}\mathcal{A}_{\Sigma}\text{-tcf})\to D^b(Qcoh_{X,T})$$
such that for a complete (and smooth) $X$ the composition $\phi \cdot {\bf K}\cdot \psi$ sends a locally free sheaf $\mathcal{F}\in coh_{X,T}$ to the complex
$Cousin(\omega _X[n]\otimes _{\mathcal{O}_X}\mathcal{F})$ - the equivariant Cousin complex, which is an equivariant resoluiton of the shifted sheaf $\omega _X[n]\otimes _{\mathcal{O}_X}\mathcal{F}$. In sum we obtain the commutative diagram of functors (Theorem \ref{main theorem nice formulation})
\[
\begin{CD}
D^b(coh _{X,T}) @>S>> D^b(coh _{X,T})\\
@VV\psi V  @AA\phi A \\
D^b(\mathcal{A}_{\Sigma}\text{-mod}) @>{\bf K} >> D^b(\text{co-}\mathcal{A}_{\Sigma}\text{-tcf})
\end{CD}
\]
where $S:D^b(coh_{X,T})\to D^b(coh_{X,T})$ is the Serre autoequivalence.

The sheaf $\mathcal{A}_{\Sigma}$ may be replaced with a "smaller" one
$\mathcal{B}_\Sigma$, which consists of certain functions on the cones $\sigma \in \Sigma$. Namely, the stalk $\mathcal{B}_{\Sigma ,\sigma}$ is the group algebra of the multiplicative semigroup $\{\exp{(l)}\}$, where $l:\sigma \to \bbC$ is a linear function which takes nonnegative integer values on all elements  $n\in N \cap \sigma$. The corresponding categories of modules are equivalent
$$\mathcal{A}_{\Sigma}\text{-mod}\simeq \mathcal{B}_{\Sigma}\text{-mod}$$

\begin{remark} It was shown in \cite{BrLu} that (a version of) the combinatorial Koszul duality functor $K$ has a geometrical meaning in terms of the categories (1) and (2) above. It is remarkable that the same combinatorial duality also has a geometrical meaning in the context of the category (3).
\end{remark}

\begin{remark} We hope that our results for smooth toric varieties can be extended to describe the category $D^b(coh _{Y,G})$, where $G$ is an algebraic group and $Y$ is an equivariant compactification of $G$ with finitely many orbits.
\end{remark}

\subsubsection{Organization of the paper} We recall some basic facts about toric varieties defined by a fan in Section \ref{term and notat}. The next Section
\ref{combinat sheaves} contains a discussion of the "combinatorial sheaves", i.e. sheaves of $\mathcal{A}_\Sigma$- and $\mathcal{B}_\Sigma$-modules. Section \ref{cosh and kos dual}  introduces combinatorial cosheaves of modules and contains a proof of  the Koszul duality theorem in this setting following \cite{BrLu}. In Section \ref{equv sheaves on toric var} we study the geometric categories of equivariant sheaves $coh_{X,T}$ and $\mathcal{O}_{X,T}\text{-mod}$ and establish their combinatorial description. Section \ref{kos dual and serre funct} relates combinatorial Koszul duality with the Serre functor on the category $D^b(coh_{X,T})$. Finally in the Appendix we recall the local cohomology and the equivariant Cousin complex.

\section{Preliminaries on toric varieties}\label{term and notat}
We recall the necessary facts about toric varieties defined by a fan \cite{Danilov}, \cite{Fulton}.

All varieties are defined over $\bbC$ and will be considered with their Zariski topology.

Let $T$ be a complex torus of dimension $n$, $M:=Hom(T,\bbC ^*)\simeq {\bbZ}^n$ the lattice of characters of $T$ and $N=Hom(\bbC ^*,T)=Hom (M,\bbZ)$ the dual lattice of 1-parameter subgroups. We have the natural isomorphism $N_{\bbR}=Lie_{\bbR}T$ - the real part of the Lie algebra of $T$.

Let $\Sigma $ be a rational polyhedral fan in $N_{\bbR}$ defining a normal toric variety $X=X_{\Sigma}$. The torus $T$ is equivariantly embedded in $X$ as the dense open orbit.
Cones in $\Sigma$ are in bijection with the $T$-orbits in $X$.
For a cone $\sigma \in \Sigma$ let ${\bf o}={\bf o}_\sigma \subset X$ be the corresponding orbit and $T_{\bf o}= T_\sigma \subset T$ its stabilizer. Since $X$ is normal, the stabilizer $T_\sigma$ is connected, i.e. a subtorus of $T$.   The Lie algebra of $T_\sigma$ is the span of $\sigma$
in $N_{\bbR}$. The orbit ${\bf o}_\sigma$ contains a unique point $x_\sigma \in {\bf o}_\sigma$ with the property that
$$\lim_{t\to 0}\lambda (t)=x_\sigma $$
for any 1-parameter subgroup $\lambda \in N$ in the interior of the cone  $\sigma$.

A cone $\tau \in \Sigma $ is a face of $\sigma $ if and only if ${\bf o}_\sigma \subset \overline{{\bf o}_\tau}$. We denote by $st({\bf o})\subset X$ the star of the orbit ${\bf o}$, which is the union of all orbits ${\bf o}'$ such that ${\bf o}\subset \overline{{\bf o}'}$.
Thus $st({\bf o})$ is the minimal $T$-invariant open subset of $X$ which contains ${\bf o}$. We have $st({\bf o}_\tau)\subset st({\bf o}_\sigma)$ if and only if $\tau \subset \sigma$. For example if $\sigma =0$ is the origin in $N _{\bbR}$, then  ${\bf o}_0=st({\bf o}_0)=T$ is the unique open orbit in $X$.

There exists a unique equivariant projection $p_\sigma :st({\bf o}_\sigma )\to {\bf o}_\sigma $ such that $p_\sigma (x_\tau)=x_\sigma$ for every $\tau \subset \sigma$. The restriction of the projection $p_\sigma $ to the torus $T$ identifies the orbit ${\bf o}_\sigma $ with the quotient torus $T/T_\sigma$.

The fiber $X_\sigma =X_{{\bf o}_\sigma } :=p_\sigma ^{-1}(x_\sigma)\subset st({\bf o}_\sigma)$ is an affine $T_\sigma$-toric variety with the (unique) fixed point $x_\sigma$. Note that $X_\sigma$ is the closure in $st({\bf o}_\sigma)$ of the torus $T_\sigma$.

Any splitting of the exact sequence of tori
\begin{equation}\label{seq to split} 1\to T_\sigma \to T\to T/T_\sigma \to 1
\end{equation}
induces an isomorphism $st({\bf o}_\sigma )\simeq {\bf o}_\sigma \times X_\sigma$.  Note that if $X$ is smooth, then $X_\sigma$ is isomorphic to the affine space ${\bbA}^{\dim T_\sigma}$.

For any orbit ${\bf o}\subset X$ the coordinate ring $A(st({\bf o}))$ is an $M$-graded algebra, which is the group algebra $\bbC [M\cap \sigma ^\vee]$ of the semigroup $M\cap \sigma ^\vee$. (Here $\sigma ^\vee \subset M_{\bbR}$ is the dual cone of $\sigma$). If $\tau \subset \sigma$, then the open embedding $st({\bf o}_\tau) \subset st({\bf o}_\sigma)$ gives the inclusion of $M$-graded algebras $A(st({\bf o}_\sigma))\to A(st({\bf o}_\tau))$.

Also, under the open embedding $st({\bf o}_\sigma )\subset st({\bf o}_\tau)$ the subvariety $X_\tau $ becomes a closed subvariety of $X_\sigma$. So we obtain the system of locally closed subvarieties $\{X_\sigma \subset X \}_{\sigma \in \Sigma}$ and closed embeddings $i_{\tau \sigma}:X_\tau \hookrightarrow X_\sigma$ for $\tau \subset \sigma$.

The coordinate ring $A(X_\sigma )$ is graded by the lattice of characters $M_\sigma$ of the torus $T_\sigma$ and the surjective ring homomorphism $A(X_\sigma )\to A(X_\tau)$ corresponding to the closed embedding $i_{\tau \sigma}$ is compatible with the surjective homomorphism of lattices $M_\sigma \to M_\tau$. More precisely: the torus $T_\tau $ acts on the $T_\sigma$-toric variety $X_\sigma$, which makes the coordinate ring $A(X_\sigma)$ also $M_\tau$-graded, and the ring homomorphism $A(X_\sigma )\to A(X_\tau)$ is a homomorphism of $M_\tau $-graded algebras.
%In any case, all these algebras are also $M$-graded and all homomorphisms are compatible with this grading.

\subsection{Some lemmas}\label{some lemmas} We prove some technical statements that will be used in Section \ref{equv sheaves on toric var}.

Fix an orbit ${\bf o}\subset X$; $T_{\bf o}\subset T$ - the stabilizer of ${\bf o}$. Choose a splitting
$T=T^{\bf o} \times T_{\bf o}$ of the sequence \eqref{seq to split}. This induces the isomorphism $st({\bf o})\simeq {\bf o} \times X_{\bf o}$ compatible with the isomorphism $T=T^{\bf o} \times T_{\bf o}$. Let $\dim {\bf o}=d$, so that $\dim X_{\bf o}=n-d$.

Assume that the toric variety $X$ is smooth. Then $X_{\bf o}={\bbA}^{\dim T_{\bf o}}$. The algebra of functions $A=A(st({\bf o}))=A({\bf o})\otimes
A(X_{\bf o})$ is $M$-graded and we can choose homogenous coordinates $x,y$ such that $A=\bbC [x_1^{\pm},...,x_d^{\pm}]\otimes \bbC [y_1,...,y_{n-d}]$.

 % $st({\bf o})={\bf o}\times (\overline{T_{\bf o}}\cap st({\bf o}))$; (where $T^o$ acts simply transitively on $o$ and trivially on $\overline{T_o}\cap st(o)$); let $d=\dim o$; since $X$ is smooth we have $\overline{T_o}\cap st(o)\simeq {\bbA}^{n-d}$; $A:=\Gamma (st(o),\mathcal{O}_X)$ - an $M$-graded algebra; We have $A=A(o)\otimes A(\overline{T_o}\cap st(o))$, where $A(o)=\bbC [x_1,x_1^{-1},...,x_d,x_d^{-1}]$ and $A(\overline{T_o}\cap st(o))=\bbC [y_1,...,y_{n-d}]$.

Let $S\subset A$ be the multiplicative set consisting of functions that are invertible on ${\bf o}$; for $x\in {\bf o}$ let $S_x\subset A$ be the multiplicative set of functions that do not vanish at $x$; put
$\Gamma :=A[S^{-1}]$; Note that the $T$-action on the algebra $A$ extends to the $T$-action on $\Gamma$, because the set $S$ is $T$-invariant;

Let $i:{\bf o}\hookrightarrow X$ be the locally closed
inclusion. Consider the sheaf $\mathcal{O}_{X\vert {\bf o}}:=i_!i^{-1}\mathcal{O}_X$ - the extension by zero from ${\bf o}$ to $X$ of the restriction $i^{-1}\mathcal{O}_X$ (Section \ref{basics on restr of sheaves}). The $T$-action on $X$ induces the $T$-action on the space of global sections $\Gamma ({\bf o},\mathcal{O}_{X\vert {\bf o}})$.  This $T$-module  $\Gamma ({\bf o}, \mathcal{O}_{X\vert {\bf o}})$ coincides with the algebra $\Gamma$.

Clearly, the $\mathcal{O}_X$-module $\mathcal{O}_{X\vert {\bf o}}$ is generated by its global sections $\Gamma ({\bf o},\mathcal{O}_{X\vert {\bf o}})$.

\begin{lemma} \label{lemma kernel of loc} For any finitely generated ($M\text{-}$) graded $A$-module $N$ the natural map $N\to N\otimes _A\Gamma$ is injective.
\end{lemma}

\begin{proof} Since $A\to \Gamma =A[S^{-1}]$ is a localization, the kernel of the map  $N\to N\otimes _A\Gamma$ consists of $n\in N$ such that $sn=0$ for some $s\in S$. Since $N$ is graded, the support of any nonzero $n\in N$ is a union of some orbits ${\bf o}'\subset st({\bf o})$. However every element $s\in S$ has no zeroes on ${\bf o}$, hence it does not vanish completely on any $T$-orbit in $st({\bf o})$. So $sn=0$ implies that $n=0$.
\end{proof}

Let $\hat{A}$ be the completion of $A$ along the orbit ${\bf o}$, i.e. $A$ is the completion with respect to the ideal $(y_1,...,y_{n-d})\subset A$. So
$\hat{A}=\bbC[x_1^{\pm},...,x_d^{\pm}][[y_1,...,y_{n-d}]]$ and $\hat{A}$ is a flat $A$-module (Proposition 10.14 in \cite{AM}). Notice that $A\subset \Gamma \subset \hat{A}$ which is also the completion of $\Gamma $
with respect to the ideal $(y_1,...,y_{n-d})$.

\begin{lemma} \label{lemma kernel of compl} Let $N$ be a finitely generated graded $A$-module. Then the natural map $N\to \hat{N}=N\otimes _A\hat{A}$ is injective.
\end{lemma}

\begin{proof} Indeed, by Theorem 10.17 in \cite{AM} the kernel of the map $N\to \hat{N}$ consists of elements $n\in N$ such that $(1+a)n=0$ for some $a\in (y_1,...,y_{n-d})\subset A$. But then $1+a\in S$ and this implies that $n=0$ as shown in the proof of Lemma \ref{lemma kernel of loc}.
\end{proof}

Notice that the $T$-action on $A$ and $\Gamma$ extends to that on  $\hat{A}$. Hence for a finitely generated graded $A$-module $N$ the inclusions $N\hookrightarrow N\otimes _A\Gamma $ and $N\hookrightarrow \hat{N}$ are $T$-equivariant.

\begin{defn} For a $T$-module $R$ denote by $R^{lf}\subset R$ the subspace of $T$-finite elements. We call $R$ $T$-{\rm locally finite} if $R^{lf}=R$. If $R$ is a locally finite {\rm algebraic} $T$-module, then $R$ is $M$-graded and we write
$$R=\bigoplus _{m\in M}R^m$$
where $R^m$ is the $m$-th homogeneous component of $R$.
\end{defn}

For example, any graded $A$-module is a $T$-module which is locally finite.

\begin{lemma} \label{equal of loc fin elts} Let $N$ be a finitely generated graded $A$-module. Then

(a) $(N\otimes _A\Gamma )^{lf}=N$. (In particular $\Gamma ^{lf}=A$.)

(b) $(\hat{N})^{lf}=N$.
\end{lemma}

\begin{proof} Since $N\otimes _A\Gamma \subset \hat{N}$ it suffices to prove (b).

Let $0\to N_1\to N\to N_2\to 0$  be a short exact sequence of finitely generated graded $A$-modules. Then we get the exact sequences

$$0\to \hat{N}_1\to \hat{N}\to \hat{N}_2\to 0$$
and
$$0\to \hat{N}_1^{lf}\to \hat{N}^{lf}\to \hat{N}_2^{lf}$$
Hence if the claim holds for $N_1,N_2$ it also holds for $N$. So we may assume that $N$ is cyclic, $N=A/K$ for a graded ideal $K\subset A$. The nonzero images $\overline{\bf x^sy^t}$ of different monomials ${\bf x^sy^t}\in A$ form a $\bbC$-basis of $N$. These are $T$-eigenvectors corresponding to distinct $T$-characters. An element of the completion $\hat{N}$ is a possibly infinite sum $f=\sum _{s,t}a_{s,t}\overline{\bf x^sy^t}$, with $a_{s,t}\in \bbC$. Then $f\in \hat{N}^{lf}$ if and only if $f$ is a finite linear combination of the monomials $\overline{\bf x^sy^t}$, i.e.
$f\in N$.
\end{proof}

\section{Sheaves of graded algebras $\mathcal{A}_\Sigma $ and $\mathcal{B}_\Sigma$ on a fan $\Sigma$}\label{combinat sheaves}

\subsection{} We will consider the fan $\Sigma =\{\sigma\}$ as a topological space with finitely many points, where $\sigma \in \overline{\tau}$ iff $\tau $ is a face of $\sigma$ (i.e. $\tau \subset \sigma$). Identifying $\Sigma$ with the set of orbits in $X$, the topology on $\Sigma$ coincides with the quotient topology on $\overline{X}=X/T$. The subsets
$$[\sigma] =\{\tau \subset \sigma\}$$
are the irreducible open subsets of $\Sigma$. If ${\bf o}_\sigma$ is the orbit corresponding to $\sigma$, then its star $st({\bf o}_\sigma)\subset X$ is the open subset which corresponds to $[\sigma]$.

To define a sheaf $F$ on $\Sigma$ it suffices to describe its sections over the open subsets $[\sigma ]$. Thus a sheaf $F$ is given by a collection of sets $\{F_\sigma \}_{\sigma \in \Sigma}$ and maps $\alpha _{\sigma \tau}:F_\sigma \to F_\tau$ for $\tau \subset \sigma$, such that $\alpha _{\tau \xi}\cdot \alpha _{\sigma \tau}=\alpha _{\sigma \xi}$.

\subsection{Sheaf of $M$-graded algebras $\mathcal{A}_{\Sigma}$}\label{sect on sheaf A_sigma}

\begin{defn} \label{defn of sheaf a} For each $\sigma \in \Sigma$ consider the corresponding orbit ${\bf o}_\sigma \subset X$ and its star $st({\bf o}_\sigma)$. Define the sheaf of $M$-graded algebras $\mathcal{A}_\Sigma$ on $\Sigma$ by putting
$$\mathcal{A}_{\Sigma ,\sigma}:=A(st({\bf o}_\sigma))$$
\end{defn}

If we identify $\Sigma$ with the quotient space $\overline{X}$, then the sheaf $\mathcal{A}_\Sigma$ is the direct image $q_*\mathcal{O}_X$ under the quotient map $q:X\to \overline{X}$.

\begin{defn} An $\mathcal{A}_\Sigma$-module is by definition a sheaf of $M$-graded modules over the sheaf of rings $\mathcal{A}_\Sigma$. The category of $\mathcal{A}_\Sigma$-modules is denoted by $\mathcal{A}_\Sigma \text{-Mod}$. It is an abelian category.

Let $\mathcal{A}_\Sigma \text{-mod} \subset \mathcal{A}_\Sigma \text{-Mod}$ be its full subcategory consisting of objects $F=\{F_\sigma\}$ such that every $F_\sigma$ is a finitely generated $\mathcal{A}_{\Sigma ,\sigma}$-module. It is an abelian subcategory.

Let $\mathcal{A}_\Sigma \text{-coh} \subset \mathcal{A}_\Sigma \text{-mod}$ be the full abelian subcategory consisting of sheaves $F$ such that for every $\tau \subset \sigma$ the structure morphism $\mathcal{A}_{\Sigma ,\tau}\otimes _{\mathcal{A}_{\Sigma ,\sigma}}F_\sigma \to F_\tau$ is an isomorphism.
\end{defn}

The group of characters $M$ acts on the category $\mathcal{A}_\Sigma \text{-Mod}$ by changing the grading of the module: every $m\in M$ shifts the grading of $F\in \mathcal{A}_\Sigma \text{-Mod}$ by $m$, $F\mapsto F(m)$.

Given $F\in \mathcal{A}_\Sigma \text{-Mod}$ and a locally closed $Z\subset \Sigma$ we denote by $F\vert _Z$ the restriction of $F$ to $Z$. Let $F_Z$ be the extension by zero of $F\vert _Z$ to $\Sigma$. Given an open subset $U\subset \Sigma$ and $Z:=\Sigma -U$ we have the usual short exact sequence
in $\mathcal{A}_\Sigma \text{-Mod}$
$$0\to F_U\to F\to F_Z\to 0$$

\begin{lemma}\label{lemma on proj res in a-mod} (1) For each $\sigma \in \Sigma$ and each $m\in M$ the object $\mathcal{A}_\Sigma (m)_{[\sigma]}\in \mathcal{A}_\Sigma \text{-Mod}$ is projective. More precisely, for any $G\in \mathcal{A}_\Sigma \text{-Mod}$ we have
$$Hom (\mathcal{A}_\Sigma (m)_{[\sigma]},G)=G_\sigma ^m$$
(2) Every object in $\mathcal{A}_\Sigma \text{-Mod}$ (resp. in $\mathcal{A}_\Sigma \text{-mod}$)  has a finite projective resolution by sheaves which are (resp. finite) direct sums $\bigoplus_{\sigma ,m}\mathcal{A}_\Sigma (m)_{[\sigma]}$.
\end{lemma}

\begin{proof} (1) The first assertion is clear.

(2) We use the fact that for each $\sigma \in \Sigma$ every (resp. finitely generated) graded module over the algebra $\mathcal{A}_{\Sigma ,\sigma}$ has a finite resolution by (resp. finitely generated) free $\mathcal{A}_{\Sigma ,\sigma}$-modules. And then induct on the dimension of cones $\sigma$  in the support of $F\in \mathcal{A}_\Sigma \text{-Mod}$.
\end{proof}

\subsection{Sheaf of graded algebras $\mathcal{B}_{\Sigma}$}

Recall that for every $\sigma \in \Sigma$ one has the affine $T_\sigma$-toric variety $X_\sigma$, which is a closed subvariety of $st({\bf o}_\sigma)$. If $\tau \subset \sigma$, then $X_\tau$ is a closed subvariety of $X_\sigma$, so we have the surjection $A(X_\sigma )\to A(X_\tau)$. The algebra $A(X_\sigma)$ is $M_\sigma$-graded and the homomorphism $A(X_\sigma )\to A(X_\tau)$ is compatible with the surjection of the groups $M_\sigma \to M_\tau$.

\begin{defn} \label{defn of sheaf b} Define the sheaf of algebras $\mathcal{B}_\Sigma$ on $\Sigma$ by putting
$$\mathcal{B}_{\Sigma ,\sigma}:=A(X_\sigma)$$
We will refer to $\mathcal{B}_\Sigma$ as a the sheaf of graded algebras
in the sense that each stalk $\mathcal{B}_{\Sigma ,\sigma}$ is graded by its own group $M_\sigma$.
\end{defn}

\begin{defn}\label{defn of b-mod} A graded $\mathcal{B}_\Sigma$-module is a
sheaf of $\mathcal{B}_\Sigma$-modules $\mathcal{N}=\{\mathcal{N}_\sigma \}_\sigma$, where each $\mathcal{N}_\sigma$ is a $M_\sigma$-graded module and for $\tau \subset \sigma$ the structure morphism $\mathcal{B}_{\Sigma ,\tau}\otimes _{\mathcal{B}_{\Sigma ,\sigma}}\mathcal{N}_\sigma \to \mathcal{N}_\tau$ is a morphism of $M_\tau$-graded modules.

As in the case of graded $\mathcal{A}_\Sigma$-modules, we have the analogous abelian categories
$$\mathcal{B}_\Sigma\text{-coh} \subset \mathcal{B}_\Sigma\text{-mod} \subset \mathcal{B}_\Sigma\text{-Mod}$$
\end{defn}

\subsubsection{Notation}\label{notation sub} Notice that the group of characters $M$ surjects onto each group $M_\sigma$. By abuse of notation we denote the image in $M_\sigma$ of $m\in M$ also by $m$. For example, we will denote the shift of grading action of the group $M$ of the category $\mathcal{B}_{\Sigma}\text{-Mod}$ as $F\mapsto F(m)$, which means that
the grading on the stalk $F_\sigma$ is shifted by the image of $m$ in $M_\sigma$. Also, given an $M_\sigma$ graded vector space $V$ and $m\in M$ we denote by $V^m$ the homogeneous component of $V$ corresponding to the image of $m$ in $M_\sigma$.

The whole discussion of the category of graded $\mathcal{A}_\Sigma$-modules can be repeated the similar category of $\mathcal{B}_\Sigma$-modules.  The next lemma is a direct analogue of Lemma
\ref{lemma on proj res in a-mod}.

\begin{lemma} \label{lemma on proj res in b-mod} For each $\sigma \in \Sigma$ and each $m\in M$ the object $\mathcal{B}(m)_{[\sigma]}\in \mathcal{B}_\Sigma \text{-Mod}$ is projective. More precisely, for any $F\in \mathcal{B}_\Sigma \text{-Mod}$ we have
$$Hom (\mathcal{B}_\Sigma (m)_{[\sigma]},F)=F_\sigma ^m$$
Every object in $\mathcal{B}_\Sigma \text{-Mod}$ (resp. in $\mathcal{B}_\Sigma \text{-mod}$) has a finite projective resolution by sheaves which are (resp. finite) direct sums $\bigoplus_{\sigma ,m}\mathcal{B}(m)_{[\sigma]}$.
\end{lemma}

\begin{proof} The same as of Lemma \ref{lemma on proj res in a-mod}.
\end{proof}

There is a natural surjection of sheaves of graded algebras $\delta :\mathcal{A}_{\Sigma}\to \mathcal{B}_{\Sigma}$. Namely, for any two cones $\tau \subset \sigma$ we have the commutative diagram of embeddings of varieties:
$$\begin{array}{ccc} st(o_\tau) & \hookrightarrow & st(o_\sigma) \\
\uparrow & & \uparrow \\
X_\tau & \hookrightarrow & X_\sigma
\end{array}
$$
whence the induced homomorphisms of graded algebras
\begin{equation} \label{comm diag for morph redo} \begin{array}{ccc} A(st(o_\tau)) & \leftarrow & A(st(o_\sigma))\\
 \downarrow \delta _\tau & & \downarrow \delta _\sigma \\
A(X_\tau) & \leftarrow & A(X_\sigma)
\end{array}
\end{equation}

\begin{prop}\label{equiv of cat of modules}
The functor of extension of scalars
$$\delta ^*:\mathcal{A}_{\Sigma }\text{-Mod}\to \mathcal{B}_{\Sigma }\text{-Mod}$$
is an equivalence of categories. It induces the equivalences of subcategories $\mathcal{A}_{\Sigma }\text{-mod}\simeq \mathcal{B}_{\Sigma }\text{-mod}$ and $\mathcal{A}_{\Sigma }\text{-coh}\to \mathcal{B}_{\Sigma }\text{-coh}$.
\end{prop}

\begin{proof} We start with a lemma.

\begin{lemma} \label{lemma on preserv of weight spaces} Fix a cone $\sigma \in \Sigma$ and consider the surjective  homomorphism of algebras $r:A(st({\bf o}_\sigma))\to A(X_\sigma)$. Let $K$ be a graded $A(st({\bf o}_\sigma))$-module and let $r^*K\in A(X_\sigma)\text{-Mod}$ be its extension of scalars. Then

(1) For any $m\in M$ the natural projection
$$\phi :K\to r^*K$$
induces an isomorphism $\phi  :K^m\to (r^*K)^m$ of the $m$-th homogeneous components.

(2) The functor $r^*:A(st({\bf o}_\sigma))\text{-Mod} \to A(X_\sigma)\text{-Mod}$ is an equivalence of categories. It induces the equivalence of subcategories $r^*:A(st({\bf o}_\sigma))\text{-mod} \to A(X_\sigma)\text{-mod}$.
\end{lemma}

\begin{proof} Choose a complementary torus $T^\sigma$, so that
\begin{equation}\label{dec of tori} T=T_\sigma \times T^\sigma
\end{equation}
(Then we have the canonical isomorphism $T^\sigma ={\bf o}_\sigma$). This
induces an isomorphism
$$st({\bf o}_\sigma)=X_\sigma \times {\bf o}_\sigma =X_\sigma \times T^\sigma$$
and hence
$$A(st({\bf o}_\sigma))=A(X_\sigma)\otimes A(T^\sigma)$$
Also \eqref{dec of tori} gives the decomposition of groups of characters $M=M_\sigma \times M^\sigma$, $m=(m_\sigma ,m^\sigma)$ ($M_\sigma$ (resp. $M^\sigma$) is the grading group of $A(X_\sigma)$ (resp. of $A(T^\sigma)$)).

Given a graded $A(st({\bf o}_\sigma))$-module $K$ we may view it as a double graded object $K=K^{\bullet ,\bullet}$ according to the decomposition of the grading groups.
If $\epsilon :A(T^\sigma)\to \bbC$ is the augmentation map (evaluation at the identity of $T^\sigma$), then the homomorphism $r$ is
$$r(a\otimes b)=a\epsilon (b)$$
This means that the canonical projection $\phi :K\to r^*K$ restricts to an isomorphism $\phi :K^{(m_\sigma ,m^\sigma)}\to (r^*K)^{m_\sigma}$
for any $m=(m_\sigma ,m^\sigma)\in M$. This proves (1).

The functor
$$\psi :A(X_\sigma)\text{-Mod}\to A(st({\bf o}_\sigma))\text{-Mod},\quad \psi (N)=A(st({\bf o}_\sigma))\otimes _{A(X_\sigma)}N$$
is the inverse of $\phi$, because every graded $A(T^\sigma)$-module is free. This proves (2).
\end{proof}

Now we can prove Proposition \ref{equiv of cat of modules}.
Note that for any $\sigma \in \Sigma$, $m\in M$ we have
\begin{equation}\label{proj go to proj}
\delta ^*(\mathcal{A}(m)_{[\sigma]})=\mathcal{B}(m)_{[\sigma]}
\end{equation}
Using part (1) of Lemma \ref{lemma on preserv of weight spaces} and the fact that $\delta ^*$ commutes with direct sums, we find that
for any $\mathcal{N}\in \mathcal{A}_\Sigma \text{-Mod}$ and any direct sum
$\bigoplus _{\sigma,m}\mathcal{A}(m)_{[\sigma]}$ we have
\begin{equation}\label{partial ff}
\begin{array}{rcl}Hom (\bigoplus _{\sigma ,m}\mathcal{A}(m)_{[\sigma]},\mathcal{N}) & = & \prod _{\sigma ,m} Hom (\mathcal{A}(m)_{[\sigma]},\mathcal{N})\\
 & = & \prod _{\sigma ,m} \mathcal{N}^m_\sigma\\
 & = & \prod _{\sigma ,m} (\delta ^*\mathcal{N}_\sigma)^m\\
 & = & \prod _{\sigma ,m} Hom (\mathcal{B}(m)_{[\sigma]},\delta ^*\mathcal{N})\\
 & = & Hom (\bigoplus _{\sigma ,m}\mathcal{B}(m)_{[\sigma]},\delta ^*\mathcal{N})\\
 & = &  Hom (\delta^*(\bigoplus _{\sigma ,m} \mathcal{A}(m)_{[\sigma]}),\delta ^*\mathcal{N})
\end{array}
\end{equation}
This implies that $\delta ^*$ is full and faithful. Indeed, let $\mathcal{M},\mathcal{N}\in \mathcal{A}_\Sigma \text{-Mod}$. Choose an exact sequence
$$\mathcal{P}_1\to \mathcal{P}_0\to \mathcal{M}\to 0$$
where each $\mathcal{P}_1,\mathcal{P}_0$ are direct sums of sheaves  $\mathcal{A}(m)_{[\sigma]}$. This gives an exact sequence
$$0\to Hom (\mathcal{M},\mathcal{N})\to Hom (\mathcal{P}_0,\mathcal{N})\to Hom (\mathcal{P}_1,\mathcal{N})$$

On the other hand, since $\delta ^*$ is right exact we have the exact sequence
$$\delta ^*\mathcal{P}_1\to \delta ^*\mathcal{P}_0\to \delta ^*\mathcal{M}\to 0$$
hence the exact sequence
$$0\to Hom (\delta ^*\mathcal{M},\delta ^*\mathcal{N})\to Hom (\delta ^*\mathcal{P}_0,\delta ^*\mathcal{N})\to Hom (\delta ^*\mathcal{P}_1,\delta ^*\mathcal{N})$$
Now use the equality \eqref{partial ff}.

It remains to show that $\delta ^*$ is essentially surjective. Every object in $\mathcal{B}_\Sigma \text{-Mod}$ is the cokernel of a morphism
\begin{equation}\label{short res}
\mathcal{Q}_1\to \mathcal{Q}_0
\end{equation}
where $\mathcal{Q}_1,\mathcal{Q}_0$ are sums of sheaves $\mathcal{B}(m)_{[\sigma]}$. The diagram \eqref{short res} lifts uniquely to a similar diagram
\begin{equation}\label{lift of short res}
\mathcal{P}_1\to \mathcal{P}_0
\end{equation}
in $\mathcal{A}_{\Sigma}\text{-Mod}$. Then $\delta ^*$ maps the cokernel of
\eqref{lift of short res} to cokernel of \eqref{short res}. This proves the equivalence
$$\delta ^*:\mathcal{A}_\Sigma \text{-Mod} \to \mathcal{B}_\Sigma \text{-Mod}$$
The last assertions of Proposition \ref{equiv of cat of modules} are clear.
\end{proof}

\subsection{Functoriality  with respect to morphisms of smooth toric varieties}

Assume that $X_1$ and $X_2$ are two $T$-toric varieties with fans $\Sigma _1$ and $\Sigma _2$. Let $f:X_1\to X_2$ is a morphism of $T$-toric varieties. Such a morphism exists (and then is unique) if and only if any cone $\tau \in \Sigma _1$ is contained in some cone $\sigma \in \Sigma _2$. Let $f:\Sigma _1\to \Sigma _2$ be the map such that  $f(\tau )$ is the smallest cone in $\Sigma _2$ which contains $\tau $. In particular $T_{\tau}\subset T_{f(\tau )}$.

The following holds
\begin{equation}\label{containment of orbits}
f(x_\tau)=x_{f(\tau)},\quad f({\bf o}_{\tau})={\bf o}_{f(\tau)},\quad
f(st({\bf o}_{\tau }))\subset st({\bf o}_{f(\tau)}),\quad
f(X_{\tau})\subset X_{f(\tau)}
\end{equation}
which implies that for any $\tau \in \Sigma _1$ we have the commutative diagram of graded algebras
\begin{equation}\label{comm diag of a and b}
\begin{array}{ccc}
A(st({\bf o}_{f(\tau)})) & \stackrel{f^*}{\to} & A(st({\bf o}_{\tau}))\\
\downarrow & & \downarrow \\
A(X_{f(\tau)}) & \stackrel{f^*}{\to} & A(X_{\tau})
\end{array}
\end{equation}
The morphisms in \eqref{comm diag of a and b} give us morphisms of
sheaves of graded algebras
$$f^{-1}\mathcal{A}_{\Sigma _2}\to \mathcal{A}_{\Sigma _1},\quad
f^{-1}\mathcal{B}_{\Sigma _2}\to \mathcal{B}_{\Sigma _1}$$
which are compatible with the surjections
$\delta _i: \mathcal{A}_{\Sigma _i}\to \mathcal{B}_{\Sigma _i}$.

In particular, we have the morphisms of ringed spaces
$$f:(\Sigma _1,\mathcal{A}_{\Sigma _1})\to (\Sigma _2,\mathcal{A}_{\Sigma _2})\quad \text{and} \quad
f:(\Sigma _1,\mathcal{B}_{\Sigma _1})\to (\Sigma _2,\mathcal{B}_{\Sigma _2})$$
which induce a pair of adjoint functors $(f^*,f_*)$ between the categories of modules $\mathcal{A}_{\Sigma _i}\text{-Mod}$ and
$\mathcal{B}_{\Sigma _i}\text{-Mod}$ respectively.
Clearly the functors $f^*$ preserves the subcategories
$\mathcal{A}_{\Sigma _i}\text{-mod}$ and
$\mathcal{B}_{\Sigma _i}\text{-mod}$.

Commutative diagram \eqref{comm diag of a and b} implies the isomorphisms of compositions
$$\delta _1^*\cdot f^*=f^*\cdot \delta _2^* \quad \text{and} \quad
\delta _2^*\cdot f_*=f_*\cdot \delta _1^*$$

\subsection{$\mathcal{B}_\Sigma$ as a sheaf of functions on the fan}

It is convenient to interpret the stalk $\mathcal{B}_{\Sigma ,\sigma}$ as the algebra of certain functions on the cone $\sigma$.
By definition, we have $\mathcal{B}_{\Sigma ,\sigma}=A(X_\sigma)$ where $X_\sigma \subset X$ is the affine $T_\sigma$-toric variety with the unique fixed point $x_\sigma \in {\bf o}_\sigma$. So $A(X_\sigma)$ is the group algebra $\bbC [M_\sigma ^+]$ of the semi-group of characters $M_\sigma ^+\subset M_\sigma$ of $T_\sigma$ which extend regularly to $X_\sigma$. ($M_\sigma$ is the group of characters $T_\sigma$).

The embedding of tori $T_\sigma \hookrightarrow T$ gives the surjection of groups of characters $M\to M_\sigma$. The preimage of the semigroup $M_{\sigma }^+$ is the semigroup $M\cap \sigma ^\vee$.
The projection $M\cap \sigma ^\vee \to M_\sigma$ is compatible with the surjection of algebras $A(st({\bf o}_\sigma ))\to A(X_\sigma)$ in the sense that the diagram
$$\begin{array}{ccc}
A(st({\bf o}_\sigma )) & = & \bbC [M\cap \sigma ^\vee]\\
\downarrow & & \downarrow \\
A(X_\sigma) & = & \bbC [M_{\sigma }^+]
\end{array}
$$
commutes.

Explicitly, one can realize the semigroup $M_{\sigma}^+$ as follows:
$M_{\sigma }^+$ consists of functions on the cone $\sigma$ of the form $exp(l)$ where $l$ is a linear function on $\sigma$ which takes nonnegative integer values on the set $\sigma \cap N$.

If $\tau \subset \sigma$ then the structure homomorphism  $\mathcal{B}_{\Sigma ,\sigma}\to \mathcal{B}_{\Sigma ,\tau}$ coincides with the natural restriction of functions $\bbC [M_{\sigma} ^+]\to \bbC [M_{\tau }^+]$.

Since we assume that the variety $X$ is smooth, each algebra $\mathcal{B}_{\Sigma ,\sigma}$ is isomorphic to a polynomial ring.

\section{Co-sheaves of modules on $\Sigma$ and combinatorial Koszul duality}\label{cosh and kos dual}

A cosheaf $G$ on the space $\Sigma$ is defined by a collection $\{G_\sigma \}_{\sigma \in \Sigma}$ and maps $\beta _{\tau \sigma}:G_\tau \to G_\sigma$ for $\tau \subset \sigma$, such that
$\beta _{\tau \sigma} \cdot \beta _{\xi \tau}=\beta _{\xi \sigma}$.

\subsubsection{Torsion co-finite modules}
When dealing with co-sheaves of modules we will be mainly interested in {\it torsion} modules which are in addition {\it co-finite}. Let us define these notions.

\begin{defn} \label{defn tcf} For $\sigma \in \Sigma$ consider the graded algebra $A(X_\sigma)$ and let $I_\sigma \subset A(X_\sigma)$ be its maximal graded ideal. We call a graded $A(X_\sigma)$-module $N$ {\rm torsion}, if it is torsion with respect to the ideal $I_\sigma$. A torsion module is called {\rm co-finite}, if its socle
$$soc (N)=\{n\in N\mid I_\sigma n=0\}$$
is finite dimensional.

A graded $A(st({\bf o}_\sigma))$-module $M$ is called {\rm torsion co-finite} if so is the corresponding graded $A(X_\sigma)$ module
$$r^*(M)=A(X_\sigma )\otimes _{A(st({\bf o}_\sigma))}M$$
where $r:A(st({\bf o}_\sigma))\to A(X_\sigma)$ is the natural surjection.

We denote the corresponding subcategories by $A(X_\sigma)\text{-tcf}\subset
A(X_\sigma)\text{-Mod}$ and $A(st({\bf o}_\sigma))\text{-tcf}\subset A(st({\bf o}_\sigma))\text{-Mod}$ respectively.
\end{defn}

\begin{remark} \label{cor on equiv of tcf} It follows from the Definition \ref{defn tcf} and Lemma \ref{lemma on preserv of weight spaces} that the functor
$$r^*:A(st({\bf o}_\sigma))\text{-tcf}\to A(X_\sigma)\text{-tcf}$$
is an equivalence.
\end{remark}

If $V$ is a graded vector space let $V^\vee$ denote its graded dual.
This gives contravariant endofunctors
$$A(X_\sigma)\text{-Mod} \stackrel{(-)^\vee}{\longrightarrow} A(X_\sigma)\text{-Mod}\quad \text{and}\quad
A(st({\bf o}_\sigma))\text{-Mod} \stackrel{(-)^\vee}{\longrightarrow} A(st({\bf o}_\sigma))\text{-Mod}$$

\begin{lemma} \label{lem sum for one mod} (1) The functorial diagram
\begin{equation}\label{comm diag}\begin{array}{ccc}
A(st({\bf o}_\sigma))\text{-Mod} & \stackrel{(-)^\vee}{\longrightarrow} & A(st({\bf o}_\sigma))\text{-Mod} \\
r^*\downarrow & & \downarrow r^*\\
A(X_\sigma )\text{-Mod} & \stackrel{(-)^\vee}{\longrightarrow} & A(X_\sigma)\text{-Mod}
\end{array}
\end{equation}
commutes.

(2) The commutative diagram \eqref{comm diag} induces  the commutative diagram of equivalences (and anti-equivalences)
\begin{equation}\label{comm diag of equiv}
\begin{array}{ccc}
A(st({\bf o}_\sigma))\text{-mod} & \stackrel{(-)^\vee}{\longrightarrow} & A(st({\bf o}_\sigma))\text{-tcf} \\
r^*\downarrow & & \downarrow r^*\\
A(X_\sigma )\text{-mod} & \stackrel{(-)^\vee}{\longrightarrow} & A(X_\sigma)\text{-tcf}
\end{array}
\end{equation}
\end{lemma}

\begin{proof} We use notation of Lemma \ref{lemma on preserv of weight spaces} and its proof. Namely, choosing a splitting $T=T_\sigma \times T^\sigma$ we may consider each $K\in A(st({\bf o}_\sigma))\text{-Mod}$ as a bigraded object. Then for any $K\in A(st({\bf o}_\sigma))\text{-Mod}$ the natural map $\phi :K\to r^*K$ identifies the component $K^{(m_\sigma ,m^\sigma)}$ with the component $(r^*K)^{m_\sigma}$.
So we have the commutative diagram
$$\begin{array}{ccccc} K^{(m_\sigma ,m^\sigma)} & \stackrel{(-)^\vee}{\mapsto} & Hom (K^{(m_\sigma ,m^\sigma)},\bbC) & = & (K^\vee)^{(-m_\sigma, -m^\sigma)} \\
\phi \downarrow & & & & \downarrow \phi \\
(r^*K)^{m_\sigma } & \stackrel{(-)^\vee}{\mapsto} & Hom ((r^*K)^{m_\sigma },\bbC) & = &  (r^*K^\vee)^{-m_\sigma}
\end{array}
$$
which proves (1).

We know that the vertical arrows in the diagram \eqref{comm diag of equiv} are equivalences (Lemma \ref{lemma on preserv of weight spaces} are Remark \ref{cor on equiv of tcf}). So for (2) it suffices to notice that the functors $(-)^\vee$ in \eqref{comm diag of equiv} are anti-equivalences, since all the modules in question have finite dimensional homogeneous components.
\end{proof}

\subsubsection{Categories of co-sheaves of modules}

\begin{defn} A co-sheaf of $\mathcal{A}_\Sigma$-modules is a collection $\{C_\sigma \}_{\sigma \in \Sigma}$ where $C_\sigma $ is an $M$-graded $\mathcal{A}_{\Sigma,\sigma}$-module, and for each pair of cones $\tau \subset \sigma $ a morphism $\beta _{\tau \sigma}:C_\tau \to C_\sigma$ of graded $\mathcal{A}_{\Sigma ,\sigma}$-modules. These morphisms must satisfy
$$\beta_{\sigma \xi}\cdot \beta _{\tau \sigma}=\beta _{\tau \xi}$$
The abelian category of co-sheaves of $\mathcal{A}_\Sigma$-modules is denoted by $\text{co-}\mathcal{A}_{\Sigma}\text{-Mod}$.

We have the full abelian subcategory $\text{co-}\mathcal{A}_{\Sigma }\text{-tcf} \subset
\text{co-}\mathcal{A}_{\Sigma}\text{-Mod}$ consisting of objects $(C_\sigma )$ such that for each $\sigma \in \Sigma$, $C_\sigma \in A(st({\bf o}_\sigma))\text{-tcf}$.

In exactly the same way one defines the abelian category $\text{co-}\mathcal{B}_{\Sigma}\text{-Mod}$ of co-sheaves of $\mathcal{B}_\Sigma$-modules and its full abelian subcategory
$\text{co-}\mathcal{B}_{\Sigma}\text{-tcf}$
\end{defn}

We have the contravariant functors
\begin{equation}\label{def of vee for sheaves} \mathcal{A}_\Sigma \text{-Mod} \stackrel{(-)^\vee}{\longrightarrow} \text{co-}\mathcal{A}_\Sigma \text{-Mod}\quad \text{and}\quad
\mathcal{B}_\Sigma \text{-Mod} \stackrel{(-)^\vee}{\longrightarrow} \text{co-}\mathcal{B}_\Sigma \text{-Mod}
\end{equation}
which send a sheaf $F=\{F_\sigma\}$ to the co-sheaf $F^\vee :=\{F^\vee _{\sigma}\}$, where $F^\vee _\sigma$ is the graded dual of the module $F_\sigma$. The following proposition is a direct consequence of Lemma
\ref{lem sum for one mod}.

\begin{prop} (1) The functorial diagram
\begin{equation}\label{comm diag of sheaves}\begin{array}{ccc}
\mathcal{A}_\Sigma \text{-Mod} & \stackrel{(-)^\vee}{\longrightarrow} & \text{co-}\mathcal{A}_\Sigma \text{-Mod} \\
\delta ^*\downarrow & & \downarrow \delta ^*\\
\mathcal{B}_\Sigma \text{-Mod} & \stackrel{(-)^\vee}{\longrightarrow} &
\text{co-}\mathcal{B}_\Sigma \text{-Mod}
\end{array}
\end{equation}
commutes.

(2) The commutative diagram \eqref{comm diag of sheaves} induces the commutative diagram of equivalences (and anti-equivalences) of categories
\begin{equation}\label{comm diag of equiv of sheaves}\begin{array}{ccc}
\mathcal{A}_\Sigma \text{-mod} & \stackrel{(-)^\vee}{\longrightarrow} & \text{co-}\mathcal{A}_\Sigma \text{-tcf} \\
\delta ^*\downarrow & & \downarrow \delta ^*\\
\mathcal{B}_\Sigma \text{-mod} & \stackrel{(-)^\vee}{\longrightarrow} &
\text{co-}\mathcal{B}_\Sigma \text{-tcf}
\end{array}
\end{equation}
\end{prop}

\begin{proof} This follows from Lemma \ref{lem sum for one mod}.
\end{proof}

\begin{remark} The category $\mathcal{A}_\Sigma \text{-mod}$ has enough projectives. Namely the object $\mathcal{A}(m)_{[\sigma]}$ is projective for every $\sigma \in \Sigma$ and $m\in M$ (Lemma \ref{lemma on proj res in a-mod}). Hence the co-sheaves $(\mathcal{A}(m)_{[\sigma]})^\vee$ are injective objects in $\text{co-}\mathcal{A}_\Sigma \text{-tcf}$. Similarly for the categories  $\mathcal{B}_\Sigma \text{-mod}$ and
$\text{co-}\mathcal{B}_\Sigma \text{-tcf}$.
\end{remark}

Let $\sigma \in \Sigma$ and let $\mathcal{M}$ be a sheaf (resp. a co-sheaf) on $\Sigma$. We denote by $\mathcal{M}_{\{\sigma\}}$ the sheaf (resp. a co-sheaf) such that
$$(\mathcal{M}_{\{\sigma\}})_\tau =\begin{cases} \mathcal{M}_\sigma,\ \text{if $\sigma =\tau$};\\
0,\quad  \text{otherwise}
\end{cases}
$$

\subsection{Combinatorial Koszul duality}

The following theorem is a variant of the Koszul duality which was constructed in \cite{BrLu}.

\begin{thm} \label{koz dual thm} There exists a natural (covariant) equivalence of triangulated categories
$${\bf K}_A:D^b(\mathcal{A}_{\Sigma}\text{-mod})\to D^b(\text{co-}\mathcal{A}_{\Sigma}\text{-tcf})$$
It has the property that ${\bf K}_A(\mathcal{A}_{\{\sigma \}})\simeq \mathcal{A}_{[\sigma]}^\vee [\dim (\sigma)]$ and ${\bf K}_A(\mathcal{A}_{[\sigma ]})\simeq \mathcal{A}_{\{\sigma\}}^\vee [\dim (\sigma)]$ for any $\sigma \in \Sigma$.

There is a similar equivalence
$${\bf K}_B:D^b(\mathcal{B}_{\Sigma}\text{-mod})\to D^b(\text{co-}\mathcal{B}_{\Sigma}\text{-tcf})$$
with analogous properties, and the functorial diagram
$$\begin{array}{ccc}{\bf K}_A:D^b(\mathcal{A}_{\Sigma}\text{-mod}) & \to & D^b(\text{co-}\mathcal{A}_{\Sigma}\text{-tcf})\\
\delta ^*\downarrow & & \downarrow \delta ^*\\
{\bf K}_B:D^b(\mathcal{B}_{\Sigma}\text{-mod}) & \to &  D^b(\text{co-}\mathcal{B}_{\Sigma}\text{-tcf})
\end{array}
$$
commutes.
\end{thm}

We recall the proof of this theorem essentially following \cite{BrLu}.

\begin{defn} Let $\mathcal{C}$ be an additive category.

(a) A $\Sigma$-diagram in $\mathcal{C}$ is a collection $\{M_\sigma \} _{\sigma \in \Sigma}$ of objects of $\mathcal{C}$ together
with morphisms $p_{\sigma \tau} : M_\sigma \to M_\tau$ for $\tau \subset \sigma$, satisfying
$p_{\tau \xi}\cdot p_{\sigma \tau} =p_{\sigma \xi}$ whenever $\xi \subset \tau \subset \sigma$.

(b) Fix an orientation of each cone in $\Sigma$. Then every $\Sigma$-diagram $\mathcal{M}= \{M_\sigma \}$ gives rise to the corresponding cellular complex in $\mathcal{C}$:
\begin{equation}\label{cel complex}
C^\bullet (\mathcal{M}) = \bigoplus _{\dim (\sigma)=n}M_\sigma \to
\bigoplus _{\dim (\tau)=n-1}M_\tau  \to ...
\end{equation}
where the terms $M_\rho$ appear in degree $-\dim (\rho)$, and the differential
is the sum of the maps $p_{\sigma \tau}$ with $\pm $ sign depending on
whether the orientations of $\sigma $ and $\tau$ agree or not.
\end{defn}

\begin{lemma}\label{lemma-acycl-1} Let $\mathcal{M} = \{M_\sigma \}$ be a {\rm constant} $\Sigma$-diagram supported between cones $\sigma$ and $\xi$ in $\Sigma$. That is $M_\tau = M$ for a fixed $M$ if $\xi \subset \tau \subset \sigma$, and $M_\tau =0$ otherwise; for $\xi \subset \tau _1\subset \tau _2\subset \sigma$ the map $p_{\tau _2\tau _1}$ is the identity. If $\sigma \neq \xi$, then the cellular complex $C^\bullet(\mathcal{M})$ is acyclic.
\end{lemma}

\begin{proof} The complex $C^\bullet (\mathcal{M})$ is isomorphic to an augmented cellular chain complex of a closed ball of dimension $\dim(\sigma) - \dim(\xi) - 1$.
\end{proof}

We only construct the equivalence ${\bf K}_A ={\bf K}$ (the equivalence  ${\bf K}_B$ is completely analogous). Consider the $\Sigma$-diagram $\mathcal{K} = \{\mathcal{K}_\sigma\}$ in the category $\text{co-}\mathcal{A}_{\Sigma}\text{-mod}$, where
$\mathcal{K}_\sigma = \mathcal{A}_{[\sigma]}^\vee$ and the maps $p_{\sigma \tau}$ are the natural projections $\mathcal{A}_{[\sigma]}^\vee \to \mathcal{A}_{[\tau]}^\vee$. This diagram $\mathcal{K}$ defines
a covariant functor
$${\bf K}: D^b(\mathcal{A}_{\Sigma}\text{-mod}) \to D^b(\text{co-}\mathcal{A}_{\Sigma}\text{-tcf})$$
in the following way. If $\mathcal{N}=\{\mathcal{N}_\sigma\} \in \mathcal{A}_{\Sigma}\text{-mod}$ is such that each $\mathcal{N}_\sigma$ is a free $\mathcal{A}_{\Sigma ,\sigma}$, the collection
$$\mathcal{K}\otimes _{\mathcal{A}_{\Sigma}}\mathcal{N}=\{\mathcal{K}_\sigma \otimes _{\mathcal{A}_{\Sigma ,\sigma}}\mathcal{N}_\sigma\}$$
is naturally a $\Sigma$-diagram in  $\text{co-}\mathcal{A}_{\Sigma}\text{-tcf}$. Thus its cellular complex $C^\bullet (\mathcal{K}\otimes _{\mathcal{A}_\Sigma}\mathcal{N})$
is a complex in $\text{co-}\mathcal{A}_\Sigma \text{-tcf}$. Given any $\mathcal{N}^\bullet \in D^b(\mathcal{A}_{\Sigma}\text{-mod})$,  define ${\bf K}(\mathcal{N}^\bullet)\in D^b(\text{co-}\mathcal{A}_{\Sigma}\text{-tcf})$ as the totalization of the double complex $\mathcal{K}\otimes _{\mathcal{A}_\Sigma }\mathcal{P}^\bullet$, where $\mathcal{P}^\bullet \to \mathcal{N}^\bullet $ is a projective resolution.

Notice that the functor ${\bf K}$ commutes with the shift of grading action of the group $M$: ${\bf K}(\mathcal{N}^\bullet (m))={\bf K}(\mathcal{N}^\bullet )(m)$

\begin{lemma} \label{lemma half of koz d} For any cone $\sigma $, there are isomorphisms

(1) ${\bf K}(\mathcal{A}_{[\sigma ]})\simeq \mathcal{A}_{\{\sigma\}}^\vee [\dim (\sigma)]$,

(2) ${\bf K}(\mathcal{A}_{\{\sigma \}})\simeq \mathcal{A}_{[\sigma]}^\vee [\dim (\sigma)]$.
\end{lemma}

\begin{proof} The first isomorphism follows from Lemma \ref{lemma-acycl-1}.

To see the second, notice that although
$\mathcal{A}_{\{\sigma \}}$ is not a projective object in $D^b(\mathcal{A}_\Sigma \text{-mod})$, still the functor ${\bf K}$ can be applied directly to $\mathcal{A}_{\{\sigma \}}$. (This is because the only nonzero stalk (at $\sigma$) of $\mathcal{A}_{\{\sigma \}}$ is a free $\mathcal{A}_{\Sigma ,\sigma}$-module.) Then the assertion (2) follows.
\end{proof}

Let us prove that the functor ${\bf K}$ is an equivalence of triangulated
categories.

The category $D^b(\mathcal{A}_{\Sigma }\text{-mod})$  is the triangulated envelope of either all objects of the form $\mathcal{A}_{[\sigma ]}(m)$ or all objects of the form $\mathcal{A}_{\{\sigma \}}(m)$ taken over all $\sigma \in \Sigma$ and $m\in M$. Similarly,
$D^b(\text{co-}\mathcal{A}_{\Sigma}\text{-mod})$
is the triangulated envelope of all objects of the form $\mathcal{A}_{[\sigma ]}^\vee(m)$ or all objects of the form $\mathcal{A}_{\{\sigma \}}^\vee(m)$.
So it suffices (using Lemma \ref{lemma half of koz d}) to show that for any $m,n,\tau ,\xi$ the functor ${\bf K}$  induces an isomorphism
\begin{equation} {\bf K}:{\bf R} Hom _{\mathcal{A}_{\Sigma }\text{-mod}}(\mathcal{A}_{[\tau]}(m),\mathcal{A}_{\{\xi\}}(n)) \to {\bf R}Hom _{\text{co-}\mathcal{A}_{\Sigma}\text{-fct}}(\mathcal{A}_{\{\tau\}}^\vee(m),
\mathcal{A}_{[\xi]}^\vee(n))
\end{equation}
But the objects $\mathcal{A}_{[\tau]}(m)$ and $\mathcal{A}_{[\xi]}^\vee(n)$ are projective and injective respectively. Hence we need to show that the map
\begin{equation}\label{koz fntr isom on hom} {\bf K}:Hom _{\mathcal{A}_{\Sigma }\text{-mod}}(\mathcal{A}_{[\tau]}(m),\mathcal{A}_{\{\xi\}}(n)) \to Hom _{\text{co-}\mathcal{A}_{\Sigma}\text{-fct}}(\mathcal{A}_{\{\tau\}}^\vee(m),
\mathcal{A}_{[\xi]}^\vee(n))
\end{equation}
is an isomorphism. Both spaces are equal to the $(n-m)$-graded part of $\mathcal{A}_{\tau}$ if
$\tau =\xi$ and vanish otherwise. This proves Theorem \ref{koz dual thm}.

The following property of the equivalence ${\bf K}$ will be used in Section \ref{kos dual and serre funct}. 

\begin{prop}\label{addition to comb in case of complete fan} Assume that the fan $\Sigma $ is complete. Then
$${\bf K}(\mathcal{A}_\Sigma)\simeq \mathcal{A}_\Sigma ^\vee [n]$$
\end{prop}

\begin{proof} As explained in the proof of part (2) in Lemma 
\ref{lemma half of koz d}, we don't need to take a projective resolution of the sheaf $\mathcal{A}_\Sigma$ to apply the functor ${\bf K}$. Therefore 
\begin{equation}\label{analyzing case on str sheaf new}
{\bf K}(\mathcal{A}_\Sigma):\ \bigoplus _{\dim (\sigma)=n}\mathcal{A}_{[\sigma]}^\vee \stackrel{d^{-n}}{\longrightarrow} \bigoplus _{\dim (\tau)=n-1}\mathcal{A}_{[\tau]}^\vee \to \cdots
\end{equation}
where the summand $\mathcal{A}_{[\sigma]}^\vee$ appears in cohomological degree $-\dim (\sigma)$. We may assume that the orientations of the $n$-dimensional cones in $\Sigma$ are induced by a global orientation of the space $N_{\bbR}$. Then the image of the natural morphism 
$$\epsilon :\mathcal{A}_\Sigma ^\vee \to \bigoplus _{\dim (\sigma)=n}\mathcal{A}_{[\sigma]}^\vee$$
is contained in the kernel of $d^{-n}$. (This morphism is injective, because the fan $\Sigma$ is complete). So we obtain the augmented complex
$$\tilde{{\bf K}}(\mathcal{A}_\Sigma):\ \mathcal{A}_\Sigma ^\vee \stackrel{\epsilon}{\longrightarrow} \bigoplus _{\dim (\sigma)=n}\mathcal{A}_{[\sigma]}^\vee \stackrel{d^{-n}}{\longrightarrow} \bigoplus _{\dim (\tau)=n-1}\mathcal{A}_{[\tau]}^\vee \to \cdots$$
We claim that this complex is acyclic. Indeed, choose a cone $\tau \in \Sigma$. Then because the fan $\Sigma$ is complete, the stalk 
$\tilde{{\bf K}}(\mathcal{A}_\Sigma)_\tau$ of this complex computes the (shifted) augmented homology of a ball (with $\mathcal{A}^\vee_{\Sigma ,\tau}$ as the group of coefficients). Hence it is acyclic.

\end{proof}

\section{Equivariant sheaves on toric varieties}\label{equv sheaves on toric var}

\subsection{Categories $\mathcal{O}_{X,T}\text{-Mod}$, $Qcoh_{X,T}$, and $coh_{X,T}$}

Consider the usual diagram
$$\begin{matrix}T\times X\\
m\downarrow \ \  \downarrow p\\
X
\end{matrix}
$$
where $p$ is the projection and $m$ is the action morphism.

Recall that a $T$-equivariant sheaf of $\mathcal{O}_X$-modules, is a pair $(\mathcal{F},\theta)$, where $\mathcal{F}\in \mathcal{O}_{X}\text{-Mod}$ and $\theta :m^*\mathcal{F}\stackrel{\sim}{\to} p^*\mathcal{F}$ is an isomorphism in $\mathcal{O}_{T\times X}\text{-Mod}$ that satisfies the usual cocycle condition. We denote the category of such sheaves by
$\mathcal{O}_{X,T}\text{-Mod}$. It is an abelian category. By abuse of notation we will denote the equivariant sheaf $(\mathcal{F},\theta)$ simply by $\mathcal{F}$.

The abelian category $\mathcal{O}_{X,T}\text{-Mod}$ contains the full abelian subcategory of quasi-coherent (resp. coherent) sheaves $Qcoh _{X,T}$ (resp. $coh_{X,T}$). We will also be interested in a natural enlargement $\mathcal{O}_{X,T}\text{-mod}\subset \mathcal{O}_{X,T}\text{-Mod}$ of the category $coh_{X,T}$ (Definition
\ref{defn of good subcat}).

\subsection{} \label{basics on restr of sheaves} For $\mathcal{F}\in \mathcal{O}_{X,T}\text{-Mod}$ the torus $T$ acts on the space of global sections $\Gamma (X,\mathcal{F})$. Moreover, for the inclusion $i:Z\hookrightarrow X$ of any $T$-invariant locally closed subset $Z\subset X$, the space of sections
$\Gamma (Z,\mathcal{F})$ is a $T$-module. Denote by $\mathcal{F} _Z:=i_!i^{-1}\mathcal{F}$ the extension by zero to $X$ of the sheaf  $i^{-1}\mathcal{F}$ on $Z$. More precisely, the sections of $\mathcal{F}_Z$ over an open set $W\subset X$ are the sections of $i^{-1}\mathcal{F}$ over $W\cap Z$ whose support is closed in $W$ \cite{Godement}, Ch 2, Thm.2.9.2. The stalks of $\mathcal{F}_Z$ are:
$$(\mathcal{F}_Z)_x=\begin{cases} \text{$\mathcal{F}_x,$ if $x\in Z$}\\
\text{$\mathcal{F}_x=0$, if $x\notin Z$}
\end{cases}
$$
Then $\mathcal{F}_Z\in \mathcal{O}_{X,T}\text{-Mod}$. Indeed, the locally closed embedding $i:Z\hookrightarrow X$ gives a locally closed embedding $j:T\times Z\hookrightarrow T\times X$ such that the two diagrams
$$\begin{array}{ccccccccccc}
T \times Z & \stackrel{j}{\to} & T\times X & & & & & T \times Z & \stackrel{j}{\to} & T\times X\\
p\downarrow & & p\downarrow & & & & & m\downarrow & & m\downarrow \\
Z & \stackrel{i}{\to} &  X & & & & & Z & \stackrel{i}{\to} &  X
\end{array}
$$
are fiber squares. And an isomorphism of $\mathcal{O}_{T\times X}$-modules
$\theta :m^*\mathcal{F}\stackrel{\sim}{\to} p^*\mathcal{F}$ restricts to the isomorphism of $\mathcal{O}_{T\times X}$-modules
$$m^*\mathcal{F_Z}=(m^*\mathcal{F})_{T\times Z}\stackrel{\theta _Z}{\to} (p^*\mathcal{F})_{T\times Z}=p^*\mathcal{F}_Z$$
So the operation $\mathcal{F}\mapsto \mathcal{F}_Z$ is an exact endofunctor of the category $\mathcal{O}_{X,T}\text{-Mod}$.

Choose an orbit ${\bf o}\subset X$ and let $V:=st({\bf o})\backslash {\bf o}$. Then for any $\mathcal{F}\in \mathcal{O}_{X,T}\text{-Mod}$ we have the exact sequence of sheaves in $\mathcal{O}_{X,T}\text{-Mod}$
$$0\to \mathcal{F}_V\to \mathcal{F}_{st({\bf o})}\to \mathcal{F}_{{\bf o}}\to 0$$
\noindent{\bf Notation.} In case $\mathcal{F}=\mathcal{O}_X$, we will write $\mathcal{O}_{X\vert {\bf o}}$ for $\mathcal{F}_{{\bf o}}$ to distinguish it from the structure sheaf of the orbit ${\bf o}$.

%For $\mathcal{F}\in \mathcal{O}_{X,T}\text{-Mod}$ and a $T$-invariant locally closed subset $Z\subset X$, the torus $T$ acts on the space of sections
%$\Gamma (Z, \mathcal{F})$.

%\begin{defn} For a vector space $V$ with a linear $T$-action we denote by $V^{lf}\subset V$ the $T$-invariant subspace on which this action is locally finite. ...
%\end{defn}

%If this action is locally finite, then $\Gamma (Z, \mathcal{F})$ becomes an $M$-graded vector space.

\subsection{}\label{sect on def of twisting} The group of characters $M$ acts by auto-equivalences of the category $\mathcal{O}_{X,T}\text{-Mod}$, $\mathcal{N}\mapsto \mathcal{N}(m)$
for $m\in M$. Namely, for any $T$-invariant locally closed subset $Z\subset X$ the $T$-action on $\Gamma (Z, \mathcal{F}(m))$ is the $T$-action on $\Gamma (Z, \mathcal{F})$ twisted by the character $m$.

Let $j:U\hookrightarrow X$ be the embedding of a $T$-invariant open subset. Then the restriction $j^*=j^{-1}$ is the functor from $\mathcal{O}_{X,T}\text{-Mod}$ to $\mathcal{O}_{U,T}\text{-Mod}$. Assume in addition that $U$ is affine with the algebra of functions $A(U)$. The $T$-action on $A$ makes it $M$-graded. Moreover the functor of global sections $\Gamma (U,(-))$ defines an equivalence between the category $Qcoh_{U,T}$ (resp. $coh_{U,T}$) and the category of (resp. finitely generated) graded $A(U)$-modules:
$$Qcoh_{U,T}\to A(U)\text{-Mod},\quad coh_{U,T}\to A(U)\text{-mod}$$
In particular, for any $\mathcal{F}\in Qcoh_{U,T}$ the $T$-module $\Gamma (U,\mathcal{F})$ is locally finite.

\subsection{}

\begin{defn} \label{defn of good subcat} Define the full subcategory $\mathcal{O}_{X,T}\text{-mod} \subset \mathcal{O}_{X,T}\text{-Mod}$ to consist of objects $\mathcal{F}$ which are quotients (in $\mathcal{O}_{X,T}\text{-Mod}$) of finite sums of
sheaves $\bigoplus  \mathcal{O}_{st({\bf o})}(m)$,
where ${\bf o}\subset X$ is an orbit and $m\in M$.
\end{defn}

\begin{remark}  Note that the category $coh _{X,T}$ is contained in
$\mathcal{O}_{X,T}\text{-mod}$. We will prove that the category $\mathcal{O}_{X,T}\text{-mod}$ is abelian and the induced functor
\begin{equation} D^b(coh_{X,T})\to D^b(\mathcal{O}_{X,T}\text{-mod})
\end{equation}
is full and faithful (Proposition \ref{prop on der eq}).
\end{remark}

%**Change the next remarks??**

%\begin{remark} \label{first ptop of good cat}
%(1) Clearly, if $\mathcal{F}\in \mathcal{O}_{X,T}\text{-mod}$ then for any orbit $o$, the restriction $\mathcal{F}\vert _o$ also belongs to $\mathcal{O}_{X,T}\text{-mod}$.

%(2) Any object $\mathcal{F}\in \mathcal{O}_{X,T}\text{-mod}$ has a finite filtration by subsheaves $\mathcal{F}_V$ for various $T$-invariant open subsets $V\subset X$ (\ref{first subsubsec}). One can choose such a filtration so that the corresponding subquotients are $\{ \mathcal{F}\vert _o\}_{o\subset X}$.

%(3) Notice that the full subcategory $coh_{X,T}\subset \mathcal{O}_{X,T}\text{-mod}$, which consists of coherent $T$-equivariant $\mathcal{O}_X$-modules, is contained in
%$\mathcal{O}_{X,T}\text{-mod}$.
%\end{remark}

\subsection{The category $(\mathcal{O}_{X,T}\text{-}mod)_{\bf o}$}

Fix an orbit ${\bf o}$ in $X$, and put $A=A(st({\bf o}))$.

\begin{defn} Let $(\mathcal{O}_{X,T}\text{-mod}) _{\bf o}$ be the full subcategory of $\mathcal{O}_{X,T}\text{-Mod}$ consisting of objects $\mathcal{F}$ which are quotients of finite direct sums $\bigoplus \mathcal{O}_{X\vert {\bf o}}(m)$with $m\in M$.
\end{defn}

In particular, the support of every nonzero object in $(\mathcal{O}_{X,T}\text{-mod}) _{\bf o}$ is equal to ${\bf o}$.

\begin{lemma} The correspondence $\mathcal{F}\mapsto \mathcal{F}_{{\bf o}}$ defines an exact functor from $\mathcal{O}_{X,T}\text{-mod}$ to $(\mathcal{O}_{X,T}\text{-mod}) _{\bf o}$.
\end{lemma}

\begin{proof}
Let $\mathcal{F}\in \mathcal{O}_{X,T}\text{-mod}$ and
choose a surjection from a finite sum $f:\bigoplus \mathcal{O}_{st({\bf o}')}(m)\to \mathcal{F}$. Then for any orbit ${\bf o}$, we obtain the surjection
$$f_{\bf o}:\bigoplus (\mathcal{O}_{st({\bf o}')}(m))_{\bf o}\to \mathcal{F}_{\bf o}$$
It remains to notice that for any two orbits ${\bf o}_1\subset \overline{{\bf o}}_2$  we have
$$(\mathcal{O}_{st({\bf o}_1)}(m))_{{\bf o}_2}=\mathcal{O}_{X\vert {\bf o}_2}(m)$$
\end{proof}

\begin{prop} \label{main prop one orbit} Consider a short exact sequence in $\mathcal{O}_{X,T}\text{-Mod}$:\begin{equation}\label{short ex seq of sheaves} 0\to \mathcal{K}\to \mathcal{P}\to \mathcal{M}\to 0
\end{equation}
where $\mathcal{P}:=\bigoplus _{i}\mathcal{O}_{X\vert {\bf o}}(m_i)$. (In particular all three sheaves are supported on ${\bf o}$ and  $\mathcal{P}, \mathcal{M}\in  (\mathcal{O}_{X,T}\text{-mod} )_{\bf o}$). Let
\begin{equation}\label{half exact seq}
0\to \Gamma ({\bf o},\mathcal{K})\to \Gamma ({\bf o},\mathcal{P})\to \Gamma ({\bf o},\mathcal{M})
\end{equation}
be the corresponding exact sequence of sections (it is the sequence of $\Gamma =\Gamma ({\bf o},\mathcal{O}_{X\vert {\bf o}})$-modules with a $T$-action). Then the following holds.

(1) $\Gamma ({\bf o},\mathcal{P})^{lf}=\oplus _iA(m_i)$.

(2) $\Gamma ({\bf o},\mathcal{K})^{lf}=\Gamma ({\bf o},\mathcal{K})\cap \Gamma ({\bf o},\mathcal{P})^{lf}$ is a finitely generated graded $A$-module.

(3) The sheaves $\mathcal{K}, \mathcal{P}, \mathcal{M}$ are generated as $\mathcal{O}_{X\vert {\bf o}}$-modules by their global sections. They are even generated by the subspaces $\Gamma ({\bf o},-)^{lf}\subset \Gamma ({\bf o},-)$ of $T$-finite elements.

(4) The sequence
\begin{equation}\label{full exact seq} 0\to \Gamma ({\bf o},\mathcal{K})\to \Gamma ({\bf o},\mathcal{P})\to \Gamma ({\bf o},\mathcal{M})\to 0
\end{equation} is exact.

(5) The sequence
\begin{equation}\label{full exact seq of loc fin} 0\to \Gamma ({\bf o},\mathcal{K})^{lf}\to \Gamma ({\bf o},\mathcal{P})^{lf}\to \Gamma ({\bf o},\mathcal{M})^{lf}\to 0
\end{equation} is an exact sequence of finitely generated graded $A$-modules.

(6) The sequence \eqref{full exact seq} is the localization $(-)\otimes _A\Gamma$ applied to the sequence \eqref{full exact seq of loc fin}.
\end{prop}

\begin{proof} (1) Since $\Gamma ({\bf o},\mathcal{O}_{X\vert {\bf o}})=\Gamma$, the claim follows from part (a) of Lemma \ref{equal of loc fin elts}.

(2) Obvious from (1) and the fact that $A$ is noetherian.

(3) Since the sheaf $\mathcal{O}_{X\vert {\bf o}}$ is generated by the subspace $A\subset \Gamma$, it follows that $\mathcal{P}$ is generated by the subspace $\Gamma ({\bf o},\mathcal{P})^{lf}$. Hence the same is true for the sheaf $\mathcal{M}$. We only have to prove that $\mathcal{K}$ is generated by the space $\Gamma ({\bf o},\mathcal{K})^{lf}$.

Choose a point $x\in {\bf o}$ and a germ $\xi _x\in \mathcal{K}_{x}\subset \mathcal{P}_{x}$.
By part (1) it has the form $\xi _x=(a_i/s_i)$, where $a_i\in A(m_i)$ and $s_i\in S_x$ (Section \ref{some lemmas}). Multiplying $\xi _x$ by the product of $s_i$'s we may assume that $\xi _x={\bf a}_x$, where ${\bf a}=(a_i)\in \bigoplus A(m_i)= \Gamma ({\bf o},\mathcal{P})^{lf}$. It suffices to prove that ${\bf a}\in \Gamma ({\bf o},\mathcal{K})$.

We may find an open neighbourhood $x\in U\subset {\bf o}$, such that $\xi _x$ comes from a section $\xi _U\in \Gamma (U,\mathcal{K})$ and $\xi _U={\bf a}_U$. Write
\begin{equation} \label{decomp into eigenvectors}
{\bf a}=\sum_{m\in M}{\bf a}^m
\end{equation}
where ${\bf a}^m$ in the $T$-eivenvector with the character $m$, i.e. $t_*{\bf a}^m=t^m{\bf a}^m$ for any $t\in T$. For a finite collection $\{t_j\}\subset T$consider the corresponding translates $\{t_{j*}{\bf a}_U\in \Gamma (t_{j*}U,\mathcal{K})\}$ ($\mathcal{K}$ is a $T$-invariant subsheaf of $\mathcal{P}$), and let
$$V:=\cap _jt_{j*}U.$$
For a large enough collection $\{t_j\}\subset T$ the $\bbC$-span of the translates $t_{j*}{\bf a}$ will contain each summand ${\bf a}^m$ in \eqref{decomp into eigenvectors}. So for each $m\in M$ we have
${\bf a}^m_V\in \Gamma (V,\mathcal{K})$ and may assume that ${\bf a}={\bf a}^m$ for some $m\in M$.

Replacing the pair $\mathcal{K}\subset \mathcal{P}$ by the pair
$\mathcal{K}(-m)\subset \mathcal{P}(-m)$ we may assume that ${\bf a}={\bf a}^0$ is a $T$-invariant global section of $\mathcal{P}$.
It follows that the various translates $\{t_*({\bf a}_V)\}_{t\in T}$ glue together to give a global section of $\mathcal{K}$. Hence
also ${\bf a}\in \Gamma ({\bf o},\mathcal{K})$. This finishes the proof of (3).

(4) We need to prove that the map of $\Gamma $-modules
$$\phi :\Gamma ({\bf o},\mathcal{P})\to \Gamma ({\bf o},\mathcal{M})$$
is surjective. The $\mathcal{O}_{X\vert {\bf o}}$-modules  $\mathcal{P}$ and $\mathcal{M}$ are generated by global sections and for every point $x\in {\bf o}$ the surjective map of stalks
$\mathcal{P}_{x}\to \mathcal{M}_{x}$ coincides with the localized map
\begin{equation}\label{surj on stalks} \phi _x:\Gamma ({\bf o},\mathcal{P})[S_x^{-1}]\to \Gamma ({\bf o},\mathcal{M})[S_x^{-1}]
\end{equation}
Put $C:=coker(\phi)$. Consider $C$ as an $A$-module and fix $c\in C$. The surjectivity of the map $\phi _x$ implies that the localized module $C[S_x^{-1}]$ vanishes for all $x\in {\bf o}$. Therefore the support $supp(c)\subset st({\bf o})=Spec(A)$ is disjoint from ${\bf o}$. But $C$ is also an $\Gamma =A[S^{-1}]$-module and for every point $y\in st({\bf o})\backslash {\bf o}$ there exists $s\in S$ such that $s(y)=0$. Therefore $supp(c)=\emptyset$, i.e. $c=0$.

(5)  By (1) we have $P:=\oplus _iA(m_i) =\Gamma ({\bf o},\mathcal{P})^{lf}$.
Put $K:=P\cap \Gamma ({\bf o},\mathcal{K})=\Gamma ({\bf o},\mathcal{K})^{lf}$ and $L:=P/K$. We have the morphism of exact sequences
%\begin{equation}\label{prep morph}
\[\begin{CD}
0 @>>> K @>>> P @>>> L @>>> 0\\
@. @| @| @ VV\alpha V @.\\
0 @>>> \Gamma ({\bf o},\mathcal{K})^{lf} @>>> \Gamma ({\bf o},\mathcal{P})^{lf}
@>>> \Gamma ({\bf o},\mathcal{M})^{lf} @.
\end{CD}\]
and need to prove that $\alpha $ is surjective (hence an isomorphism).

Clearly, $\alpha$ is injective. It follows that the composition
$$L\otimes _A\Gamma \to \Gamma ({\bf o},\mathcal{M})^{lf}\otimes _A\Gamma \to \Gamma ({\bf o},\mathcal{M})\otimes _A\Gamma =\Gamma ({\bf o},\mathcal{M})$$
is injective. This composition is also surjective because of the commutative diagram
\begin{equation}\label{aux comm diag}
\begin{array}{ccc}
 P\otimes _A\Gamma & \to & L\otimes _A\Gamma\\
 a\downarrow & & \downarrow \\
 \Gamma ({\bf o},\mathcal{P}) & \stackrel{\phi}{\to} & \Gamma ({\bf o},\mathcal{M})
\end{array}
\end{equation}
where the map $a$ is an isomorphism and the horizontal arrows are surjections. Now the assertion follows from part (a) of Lemma \ref{equal of loc fin elts}.

(6) We have the natural commutative diagram of $\Gamma$-modules
\[\begin{CD}
0 @>>> \Gamma ({\bf o},\mathcal{K})^{lf}\otimes _A\Gamma  @>>> \Gamma ({\bf o},\mathcal{P})^{lf}\otimes _A\Gamma  @>>> \Gamma ({\bf o},\mathcal{M})^{lf}\otimes _A\Gamma  @>>> 0\\
@. @VVbV @VVaV @ VVdV @.\\
0 @>>> \Gamma ({\bf o},\mathcal{K}) @>>> \Gamma ({\bf o},\mathcal{P})
@>>> \Gamma ({\bf o},\mathcal{M}) @>>> 0
\end{CD}\]
The map $a$ is clearly an isomorphism. The map $d$ is an isomorphism as proved in (5). Hence also $b$ is an isomorphism.
\end{proof}

\begin{corollary}\label{cor that small is ab cat} The category $(\mathcal{O}_{X,T}\text{-mod})_{\bf o}$ is an abelian subcategory of $\mathcal{O}_{X,T}\text{-mod}$. It consists of all subquotients of finite direct sums $\oplus \mathcal{O}_{X\vert {\bf o}}(m_i)$.
\end{corollary}

\begin{proof} Let $\mathcal{K}\subset \bigoplus  \mathcal{O}_{X\vert {\bf o}}(m_i)$ be as in Proposition \ref{main prop one orbit}.
It suffices to show that $\mathcal{K}\in (\mathcal{O}_{X,T}\text{-mod})_{\bf o}$.
By (2) and (3) of this proposition the sheaf $\mathcal{K}$ is generated by the subspace $\Gamma ({\bf o},\mathcal{K})^{lf}$ of its global sections $\Gamma ({\bf o},\mathcal{K})$, and $\Gamma ({\bf o},\mathcal{K})^{lf}$ is a finitely generated graded $A$-module. In particular, the natural map
$$\Gamma ({\bf o},\mathcal{K})^{lf}\otimes _A\mathcal{O}_{X\vert{\bf o}}\to \mathcal{K}$$
is an isomorphism. Choose a finite sum $Q:=\oplus A(m)$ and a surjection of graded $A$-modules $Q\to \Gamma ({\bf o},\mathcal{K})^{lf}$. This induces a surjection of equivariant $\mathcal{O}_X$-modules
$$\bigoplus \mathcal{O}_{X\vert {\bf o}}(n) \simeq
Q\otimes _A\mathcal{O}_{X\vert {\bf o}}\to
\Gamma ({\bf o},\mathcal{K})^{lf}\otimes _A\mathcal{O}_{X\vert {\bf o}}\simeq \mathcal{K}$$
So $\mathcal{K}\in
(\mathcal{O}_{X,T}\text{-mod})\vert _{\bf o}$. Thus every subquotient of a finite sum
$\bigoplus \mathcal{O}_{X\vert {\bf o}}(m_i)$ is in $(\mathcal{O}_{X,T}\text{-mod})\vert _{\bf o}$, and these form an abelian subcategory of $\mathcal{O}_{X,T}\text{-Mod}$.
\end{proof}

\begin{corollary}\label{equiv of cat for one orbit} The abelian category $(\mathcal{O}_{X,T}\text{-mod}) _{\bf o}$ is equivalent to the category $A\text{-mod}$   of finitely generated graded $A$-modules. Namely the functors
$$\mathcal{L}: L\mapsto L\otimes _A\mathcal{O}_{X\vert {\bf o}}\quad \text{and}\quad \gamma :\mathcal{M}\mapsto \Gamma ({\bf o},\mathcal{M})^{lf}$$
are exact and are the mutually inverse equivalences of categories.
\end{corollary}

\begin{proof} Let $L\in A\text{-mod}$ and $\mathcal{M}=L\otimes _A\mathcal{O}_{X\vert {\bf o}}$. We claim that the natural map of $A$-modules
$$\tau: L\to \Gamma ({\bf o},\mathcal{M})^{lf}$$
is an isomorphism.

First we claim that $\tau $ is injective. It suffices to show that the natural composition of maps
\begin{equation}\label{comp of maps}
L\to L\otimes _A\Gamma \to \Gamma ({\bf o},\mathcal{M})
\end{equation}
is injective. The first map in \eqref{comp of maps} is injective by Lemma \ref{lemma kernel of loc}. The vanishing of the kernel of the second one is established in the same way as the vanishing of the $\Gamma$-module $C$ in the proof of (4) of Proposition \ref{main prop one orbit}.

To prove that $\tau$ is surjective, choose a finite sum $Q:=\oplus A(m)$ and a surjection of graded $A$-modules $Q\to L$. This induces the surjection in $(\mathcal{O}_{X,T}\text{-mod})\vert _{\bf o}$:
$$Q\otimes \mathcal{O}_{X\vert {\bf o}}\to \mathcal{M}$$
We have the commutative diagram of $A$-modules
\[\begin{CD}
Q @>>> L \\
@VVV @ VV\tau V\\
\Gamma ({\bf o},Q\otimes \mathcal{O}_{X\vert {\bf o}})^{lf} @>>>
\Gamma ({\bf o},\mathcal{M})^{lf}
\end{CD}\]
where the left vertical arrow is an isomorphism by Part (1) in
Proposition \ref{main prop one orbit} and the lower horizontal arrow is surjective by part (5) of that Proposition. Hence $\tau$ is surjective.

Vice versa, if $\mathcal{M}\in
(\mathcal{O}_{X,T}\text{-mod})\vert _{\bf o}$, then by part (3) of Proposition \ref{main prop one orbit} the
subspace $\Gamma ({\bf o},\mathcal{M})^{lf}\subset \Gamma ({\bf o},\mathcal{M})$ is in $A\text{-mod}$ and generates the $\mathcal{O}_{X\vert {\bf o}}$-module
$\mathcal{M}$, i.e. the natural map
$$\Gamma ({\bf o},\mathcal{M})^{lf}\otimes _A\mathcal{O}_{X\vert {\bf o}}\to \mathcal{M}$$
is an isomorphism.
\end{proof}

As we noted in Section \ref{sect on def of twisting} the category $A\text{-mod}$ is also equivalent to the category $coh_{st({\bf o}),T}$ of coherent $T$-equivariant sheaves on $st({\bf o})$. This equivalence is given by the localization functor
$L\mapsto L\otimes _A\mathcal{O}_{st({\bf o})}$. The composition of this equivalence with the functor
$$coh_{st({\bf o}),T}\to (\mathcal{O}_{X,T}\text{-mod}) _{\bf o},\quad \mathcal{F}\mapsto \mathcal{F}_{\bf o}$$ is the equivalence in Corollary \ref{equiv of cat for one orbit}.

Thus we obtain the following corollary.

\begin{corollary} The functor
$$coh_{st({\bf o}),T}\to (\mathcal{O}_{X,T}\text{-mod}) _{\bf o},\quad \mathcal{F}\mapsto \mathcal{F}_{\bf o}$$
is an equivalence of categories.
\end{corollary}

\subsection{Study of the category $\mathcal{O}_{X,T}\text{-mod}$}

\begin{prop} \label{prop on general abelian subcat} The category $\mathcal{O}_{X,T}\text{-mod}$ is an abelian subcategory of $\mathcal{O}_{X,T}\text{-Mod}$.
\end{prop}

\begin{proof} Clearly the subcategory $\mathcal{O}_{X,T}\text{-mod}$ is closed under quotients. It suffices to prove that any subobject  $\mathcal{K}$ of a finite direct sum $\mathcal{P}:=\bigoplus _{{\bf o},m} \mathcal{O}_{st(o)}(m)$ belongs to $\mathcal{O}_{X,T}\text{-mod}$. For each orbit ${\bf o}$ the sheaf $\mathcal{K}_{st({\bf o})}$ is a subobject of $\mathcal{K}$. Hence $\mathcal{K}$ is the quotient of the direct sum $\bigoplus \mathcal{K}_{st({\bf o})}$ and we may assume that for a fixed orbit ${\bf o}\subset X$
$$\mathcal{K}=\mathcal{K}_{st({\bf o})}\subset \mathcal{P}_{st({\bf o})}$$
The sheaf $\mathcal{P}_{st({\bf o})}$ is a finite sum of sheaves
$\mathcal{O}_{st({\bf o}')}(m)$ for some orbits ${\bf o}'\subset st({\bf o})$.

First assume that ${\bf o}=T$ is the dense orbit. Then $\mathcal{K}_{st({\bf o})}=\mathcal{K}_{\bf o}$, $\mathcal{O}_{st({\bf o})}=\mathcal{O}_{X\vert {\bf o}}$ and so $\mathcal{K}_{st({\bf o})}$ is in $(\mathcal{O}_{X,T}\text{-mod})_{\bf o}\subset \mathcal{O}_{X,T}\text{-mod}$ by Corollary \ref{cor that small is ab cat}.

For a general orbit ${\bf o}$, let $V=st({\bf o})\backslash {\bf o}\subset st({\bf o})$, and consider the exact sequence in
$\mathcal{O}_{X,T}\text{-Mod}$
\begin{equation} 0\to \mathcal{K}_V\to \mathcal{K}\stackrel{\theta}\to \mathcal{K}_{\bf o}\to 0
\end{equation}
It suffices to find a morphism $\mathcal{Q}:=\bigoplus _n \mathcal{O}_{st({\bf o})}(n)\stackrel{q}{\to }\mathcal{K}$ such that the composition $\theta \cdot q :\mathcal{Q}\to \mathcal{K}_{\bf o}$ is surjective. Indeed, we may assume by induction on the number of orbits in $V$, that there exists a surjection from a finite direct sum
$$\mathcal{R}:=\bigoplus _{{\bf o}',m'}\mathcal{O}_{st({\bf o}')}(m')\stackrel{r}{\longrightarrow} \mathcal{K}_V$$
Then the map $(r,q):\mathcal{R}\oplus \mathcal{Q}\to \mathcal{K}$ is surjective and so $\mathcal{K}\in \mathcal{O}_{X,T}\text{-mod}$.

To find the desired morphism $q:\mathcal{Q}\to \mathcal{K}$ as above consider the embedding $\mathcal{K}_{\bf o}\subset \mathcal{P}_{\bf o}$. This is the situation of Proposition \ref{main prop one orbit} above. Recall that
$\Gamma ({\bf o},\mathcal{P}_{\bf o})^{lf}=P:=\bigoplus A(m)$ where $A=A(st({\bf o}))$.
That is the natural morphism $\mathcal{P}\to \mathcal{P}_{\bf o}$ induces an isomorphism $\Gamma (st({\bf o}),\mathcal{P})\to \Gamma ({\bf o},\mathcal{P}_{\bf o})^{lf}$.
The sheaf $\mathcal{K}_{\bf o}$ are generated (as $\mathcal{O}_X$-module) by the subspace $K:=\Gamma ({\bf o},\mathcal{K}_{\bf o})^{lf}=\Gamma ({\bf o},\mathcal{K}_{\bf o})\cap P$ (Prop. \ref{main prop one orbit}).
Choose a surjection of finitely generated graded $A$-modules $\bigoplus A(n)\to K$ and put
$$\mathcal{Q}:=\left(\bigoplus A(n)\right)\otimes _A \mathcal{O}_{st({\bf o})}=\bigoplus \mathcal{O}_{st({\bf o})}(n)$$

We have the commutative diagram in $\mathcal{O}_{X,T}\text{-Mod}$:
$$\begin{array}{ccc}
\mathcal{Q} & & \\
\downarrow & & \\
K\otimes _A\mathcal{O}_{st({\bf o})} & \hookrightarrow & \mathcal{P} \\
\downarrow & & \downarrow \\
\mathcal{K}_{\bf o} & \hookrightarrow & \mathcal{P}_{\bf o}
\end{array}
$$
where the composition $\mathcal{Q}\to \mathcal{K}_{\bf o}$ is surjective. So it suffices to prove that the subspace $K\subset \Gamma (st({\bf o}),\mathcal{P})$ belongs to $\Gamma (st({\bf o}),\mathcal{K})$, i.e. the map
$$\delta :\Gamma (st({\bf o}),\mathcal{K})^{lf}\to \Gamma ({\bf o},\mathcal{K}_{\bf o})^{lf}=K$$
is surjective.

Let ${\bf a}=\sum_{m\in M}{\bf a}^m\in K\subset P$. It suffices to prove that each $T$-eigenvector ${\bf a}^m$ comes from a section in $\Gamma (st({\bf o}),\mathcal{K})^{lf}$. Replacing the pair $\mathcal{K}\subset \mathcal{P}$ by the pair $\mathcal{K}(-m)\subset \mathcal{P}(-m)$
 we may assume that $m=0$, i.e. ${\bf a}^0\subset K$ is a $T$-invariant
element in $\Gamma ({\bf o},\mathcal{K}_{\bf o})^{lf}$. It comes from a section
${\bf a}_U\in \Gamma (U,\mathcal{K})$ for an open subset ${\bf o}\subset U\subset st({\bf o})$. Then $U$ intersects every orbit ${\bf o}'\subset st({\bf o})$. Therefore the $T$-invariant section ${\bf a}_U$ comes from an invariant section in
$\Gamma (st({\bf o}),\mathcal{K})$.  This shows that the map $\delta$ is surjective and finishes the proof of Proposition \ref{prop on general abelian subcat}.
\end{proof}

\begin{lemma}\label{gl sect on st and o} Let $\mathcal{M}\in \mathcal{O}_{X,T}\text{-mod}$. Fix an orbit ${\bf o}\subset X$ and let $V:=st({\bf o})\backslash {\bf o}$. Consider the short exact sequence in $\mathcal{O}_{X,T}\text{-Mod}$
$$0\to \mathcal{M}_V\to \mathcal{M}_{st({\bf o})}\to \mathcal{M}_{\bf o}\to 0$$
Then the map $\Gamma (st({\bf o}),\mathcal{M}_{st({\bf o})})^{lf}\stackrel{\beta}{\to} \Gamma ({\bf o},\mathcal{M}_{\bf o})^{lf}$ is an isomorphism.
\end{lemma}

\begin{proof} Since $\Gamma (st({\bf o}),\mathcal{M}_V)=0$ it suffices to show that $\beta $ is surjective. Choose a surjection $\mathcal{P}:=\bigoplus _{{\bf o}',m}\mathcal{O}_{st({\bf o}')}(m)\to \mathcal{M}$. This gives the commutative diagram of equivariant sheaves
\[
\begin{CD}
\mathcal{P}_{st({\bf o})} @>>> \mathcal{M}_{st({\bf o})}\\
@VVV  @VVV \\
\mathcal{P}_{\bf o} @>>> \mathcal{M}_{\bf o}
\end{CD}
\]
and the corresponding commutative diagram of locally finite $T$-modules
\[
\begin{CD}
\Gamma (st ({\bf o}), \mathcal{P}_{st({\bf o})})^{lf} @>>> \Gamma (st({\bf o}), \mathcal{M}_{st({\bf o})})^{lf}\\
@VV\alpha V  @VV\beta V \\
\Gamma ({\bf o},\mathcal{P}_{\bf o})^{lf} @>\gamma >> \Gamma ({\bf o}, \mathcal{M}_{\bf o})^{lf}
\end{CD}
\]
Then the map $\alpha$ is an isomorphism and $\gamma $ is surjective by parts (1) and (5) of Proposition \ref{main prop one orbit} respectively.
Hence $\beta$ is also surjective.
\end{proof}

\begin{corollary} \label{last cor before equiv} (1) For any orbit ${\bf o}\subset X$ the functor $\Gamma (st({\bf o}),-)^{lf}$ is an exact functor
$$\Gamma (st({\bf o}),-)^{lf}:\mathcal{O}_{X,T}\text{-mod}\to A(st({\bf o}))\text{-mod}$$

(2) The sheaves $\{\mathcal{O}_{st({\bf o})}(m)\}_{{\bf o},m}$ are projective objects in $\mathcal{O}_{X,T}\text{-mod}$.

(3) Any object $\mathcal{M}\in \mathcal{O}_{X,T}\text{-mod}$ has a finite left (projective) resolution
$$\cdots \to \mathcal{P}_{-1}\to \mathcal{P}_0\to \mathcal{M}\to 0$$
where each $\mathcal{P}_{i}$ is a finite sum of sheaves in $\{\mathcal{O}_{st({\bf o})}(m)\}_{{\bf o},m}$.
\end{corollary}

\begin{proof} (1) Consider the functor
$$\mathcal{O}_{X,T}\text{-mod}\to (\mathcal{O}_{X,T}\text{-mod})_{\bf o},\quad \quad
\mathcal{M}\mapsto \mathcal{M}_{\bf o}$$
By Lemma \ref{gl sect on st and o} the functors $\Gamma (st({\bf o}),\mathcal{M})^{lf}$ and $\Gamma ({\bf o},\mathcal{M}_{\bf o})^{lf}$ are isomorphic on the category $\mathcal{O}_{X,T}\text{-mod}$. So it remains to prove the following statement: let $\mathcal{M}\to \mathcal{M}'$ be a surjection in $\mathcal{O}_{X,T}\text{-mod}$. Then the induced map
$$\Gamma ({\bf o},\mathcal{M}_{\bf o})^{lf}\stackrel{\mu}{\to }\Gamma ({\bf o},\mathcal{M}'_{\bf o})^{lf}$$
is a surjection of finitely generated graded $A(st({\bf o}))$-modules.

Choose a surjective map $\mathcal{P}=\bigoplus \mathcal{O}_{st({\bf o})}(m)\to \mathcal{M}$. By parts (1) and (5) of Proposition \ref{main prop one orbit} the induced maps
$$\Gamma ({\bf o},\mathcal{P}_{\bf o})^{lf}\to \Gamma ({\bf o},\mathcal{M}_{\bf o})^{lf}\quad \text{and}\quad \Gamma ({\bf o},\mathcal{P}_{\bf o})^{lf}\to \Gamma ({\bf o},\mathcal{M}'_{\bf o})^{lf}$$
are surjections of finitely generated graded $A(st({\bf o}))$-modules. Hence also $\mu$ is such.

(2) The functors $Hom (\mathcal{O}_{st({\bf o})}(m),(-))$ and $\Gamma (st({\bf o}),(-))^m$ are isomorphic on $\mathcal{O}_{X,T}\text{-mod}$. The second functor is exact by part (1), hence the object $\mathcal{O}_{st({\bf o})}(m)$ is projective.

(3) Let $\mathcal{M}\in \mathcal{O}_{X,T}\text{-mod}$. The existence of a possibly infinite resolution
$$\cdots \to \mathcal{P}_{-1}\to \mathcal{P}_0\to \mathcal{M}\to 0$$
where each $\mathcal{P}_{i}$ is a finite sum of sheaves in $\{\mathcal{O}_{st({\bf o})}(m)\}_{{\bf o},m}$ follows from the definition of the category $\mathcal{O}_{X,T}\text{-mod}$ and the fact that it is abelian. To show that one can choose a finite resolution we can argue as in the proof Lemma \ref{lemma on proj res in a-mod} or else use Proposition \ref{prop on comb equiv} below.

\end{proof}

\subsection{The functor $D^b(coh_{X,T})\to D^b(\mathcal{O}_{X,T}\text{-mod})$ is full and faithful}

\begin{prop} \label{prop on der eq} The embedding of abelian categories
$coh_{X,T}\to \mathcal{O}_{X,T}\text{-mod}$ induces the fully faithful
functor $D^b(coh_{X,T})\to D^b(\mathcal{O}_{X,T}\text{-mod})$.
\end{prop}

\begin{proof} Since $X$ is smooth, every object in $coh_{X,T}$ is quasi-isomorphic to a finite complex of equivariant vector bundles, i.e. locally free equivariant $\mathcal{O}_{X}$-modules of finite rank. Let $F$ and $G$ be equivariant vector bundles on $X$. It suffices to prove that the natural map
\begin{equation}\label{map of exts} Ext ^\bullet _{coh_{X,T}}(F,G)\to Ext ^\bullet _{\mathcal{O}_{X,T}\text{-mod}}(F,G)
\end{equation}
is an isomorphism.

To compute the right hand side of \eqref{map of exts} choose a  projective resolution
\begin{equation}\label{resol to prove eq} \cdots \to \mathcal{P}_{-1} \to \mathcal{P}_0 \to F\to 0
\end{equation}
as in part (3) of Corollary \ref{last cor before equiv}. That is, the equivariant sheaves $\mathcal{P}_{i}$ consist of finite direct sums of sheaves $\{\mathcal{O}_{st({\bf o})}(m)\}_{{\bf o},m}$. Then
$$Ext ^i _{\mathcal{O}_{X,T}\text{-mod}}(F,G)=H^i (Hom  _{\mathcal{O}_X} (\mathcal{P}_\bullet ,G)^T)=H^i (\Gamma (X,\mathcal{H}om _{\mathcal{O}_X}(\mathcal{P}_\bullet ,G))^T)$$
On the other hand notice that the complex of sheaves of $\mathcal{O}_X$-modules
\begin{equation}0\to \mathcal{H}om _{\mathcal{O}_X}(F,G)\to
\mathcal{H}om _{\mathcal{O}_X}(\mathcal{P}_0,G)\to \cdots
\end{equation}
is exact (because the stalks of all sheaves in the complex \eqref{resol to prove eq} are free modules). Also notice that the sheaves
$\mathcal{H}om _{\mathcal{O}_X}(\mathcal{P}_i,G)$ are finite direct sums
$$\mathcal{H}om _{\mathcal{O}_X}(\mathcal{O}_{st({\bf o})}(m),G)=j_*G_{st({\bf o})}(-m)$$
where $j:st ({\bf o})\hookrightarrow X$ is the open embedding.
It follows that the sheaves $\mathcal{H}om _{\mathcal{O}_X}(\mathcal{P}_i,G)$ are quasi-coherent equivariant sheaves which are acyclic for the functor $\Gamma (X,-)^T$ (because the open subsets $st({\bf o})\subset X$ are affine). By Lemma \ref{der of coh is full and faith in der of qcoh} we have
$$Ext ^i _{coh_{X,T}}(F,G)=Ext ^i _{Qcoh_{X,T}}(F,G)$$
Hence
$$Ext ^i _{coh_{X,T}}(F,G)={\bf R}^i\Gamma (X,\mathcal{H}om _{\mathcal{O}_X}(F,G))^T=H^i (\Gamma (X ,\mathcal{H}om _{\mathcal{O}_X}(\mathcal{P}_\bullet,G))^T)$$
which proves the proposition.
\end{proof}

\begin{lemma}\label{der of coh is full and faith in der of qcoh}
The natural functor $D^b(coh_{X,T})\rightarrow D^b(Qcoh_{X,T})$ is full and faithful.
\end{lemma}

\begin{proof} The fact that the non-equivariant functor
$D^b(coh_{X})\rightarrow D^b(Qcoh_{X})$ is full and faithful is well known and follows from the fact that every quasi-coherent sheaf is a union of its coherent subsheaves.

In the equivariant situation the same holds: let $\mathcal{F}\in Qcoh_{X,T}$ be a union of its (non-equivariant) subsheaves $\{\mathcal{F}_i\}$. Fix one $\mathcal{F}_i$ and let $\mathcal{H}\subset \mathcal{F}$ be an $\mathcal{O}_X$-submodule of $\mathcal{F}$ generated by the shifts $\{t_*\mathcal{F}_i\}\vert _{t\in T}$. Clearly $\mathcal{H}$ is a $T$-equivariant quasi-coherent subsheaf of $\mathcal{F}$. It remains to show that $\mathcal{H}$ is coherent. Let $U\subset X$ be a $T$-invariant open subset. The $A(U)$-module $\mathcal{F}_i(U)$ is generated by a finite set of sections $\{s_j\}\subset  \mathcal{F}_i(U)$. Then the $A(U)$-module $\mathcal{H}(U)$ is generated by the shifts $\{t_*s_j\}\vert _{t\in T}$, which are contained in a finite dimensional vector space, because the $T$-action on $\mathcal{F}(U)$ is locally finite.
\end{proof}

\subsection{Combinatorial description of the category $\mathcal{O}_{X,T}\text{-mod}$}

%As we already remarked, for every object $\mathcal{F}\in \mathcal{O}_{X,T}\text{-mod}$ we have $\mathcal{F}\vert _o\in (\mathcal{O}_{X,T}\text{-mod})_o$ for each orbit $o\subset X$. In addition $\mathcal{F}$ has a finite filtration with subquotients $\{\mathcal{F}\vert _o\}_{o\subset X}$.

Our goal is to extend the description of the category
$(\mathcal{O}_{X,T}\text{-mod})_{\bf o}$ in Corollary \ref{equiv of cat for one orbit} to that of $\mathcal{O}_{X,T}\text{-mod}$.

\subsubsection{}
Recall the topological space $\Sigma \simeq \overline{X}:=X/T$, and $q:X\to \overline{X}$ - the quotient map. Then the sheaf of graded algebras $\mathcal{A}_\Sigma$ is the direct image $q_*\mathcal{O}_X$
(Section \ref{sect on sheaf A_sigma}).

Recall that for any orbit ${\bf o}\subset X$ and any $\mathcal{M}\in \mathcal{O}_{X,T}\text{-mod}$ the space $\Gamma (st({\bf o}),\mathcal{M})^{lf}$ is a finitely generated graded $A(st({\bf o}))$-module (Corollary \ref{last cor before equiv}). This implies that we have the functor
\begin{equation} \label{defn of the functor theta} \theta :\mathcal{O}_{X,T}\text{-mod}\to
\mathcal{A}_\Sigma \text{-mod},
\quad \text{where}\quad \theta (\mathcal{M})_\sigma =\Gamma (st({\bf o}_\sigma ),\mathcal{M})^{lf}
\end{equation}

\begin{prop} \label{prop on comb equiv} The functor $\theta :\mathcal{O}_{X,T}\text{-mod}\to
\mathcal{A}_\Sigma \text{-mod}$ is an equivalence of abelian categories. It induces the equivalence of full subcategories
$\theta :coh _{X,T}\to \mathcal{A}_{\Sigma }\text{-coh}$
\end{prop}

\begin{proof} By Corollary \ref{last cor before equiv} we know that
the functor $\theta $ is exact. The category $\mathcal{O}_{X,T}\text{-mod}$ (resp. $\mathcal{A}_\Sigma \text{-mod}$) has a system of projective generators $\{\mathcal{O}_{st({\bf o}_\sigma)}(m)\}_{\sigma ,m}$ (resp. $\{\mathcal{A}_{[\sigma]}(m)\}_{\sigma ,m}$). It is clear that
$$\theta (\mathcal{O}_{st({\bf o}_\sigma)}(m))=\mathcal{A}_{[\sigma]}(m)$$

For any $\mathcal{M}\in \mathcal{O}_{X,T}\text{-mod}$ and ${\bf o}\subset X$ we have a functorial isomorphism
\begin{equation}\label{isom of homs} Hom (\mathcal{O}_{st({\bf o})}(m),\mathcal{M})=\Gamma (st({\bf o}),\mathcal{M})^m
\end{equation}
Likewise, for any $\sigma \in \Sigma$ and any $\mathcal{N}\in \mathcal{A}_{\Sigma}\text{-mod}$ we have
\begin{equation}\label{second isom of homs}Hom (\mathcal{A}_{[\sigma]}(m),\mathcal{N})=\mathcal{N}_{\sigma}^m
\end{equation}
It follows that $\theta$ induces an isomorphism
$$\begin{array}{rcl}Hom (\mathcal{O}_{st({\bf o}_\sigma)}(m),\mathcal{M})& = & \Gamma (st({\bf o}_\sigma),\mathcal{M})^m\\
& = & \theta (\mathcal{M})_\sigma ^m\\
& = & Hom (\mathcal{A}_{[\sigma]}(m),\theta(\mathcal{M}))\\
& = & Hom (\theta(\mathcal{O}_{st({\bf o}_\sigma)}(m)),\theta(\mathcal{M}))
\end{array}
$$
Now arguing as in the proof of Proposition \ref{equiv of cat of modules}, one shows that $\theta$ is full and faithful and that it is an equivalence of categories. The last assertion of the proposition is clear.
\end{proof}

Combining the equivalence $\theta$ with the equivalence $\delta ^*$ in Proposition \ref{equiv of cat of modules}
we find the following corollary

\begin{corollary} The composition  $\delta ^*\cdot \theta :\mathcal{O}_{X,T}\text{-mod}\to \mathcal{B}_{\Sigma}\text{-mod}$
is an equivalence of abelian categories. It induces an equivalence
$coh_{X,T}\simeq \mathcal{B}_{\Sigma}\text{-coh}$.
\end{corollary}

By Proposition \ref{prop on der eq} we obtain the following theorem.

\begin{thm} \label{summary on der equiv and emb} (1) We have fully faithful functors
$$D^b(coh_{X,T})\hookrightarrow D^b(\mathcal{A}_{\Sigma}\text{-mod})\quad \text{and}\quad
D^b(coh_{X,T})\hookrightarrow D^b(\mathcal{B}_{\Sigma}\text{-mod})$$
whose essential images are the categories $D^b(\mathcal{A}_{\Sigma}\text{-coh})$\quad \text{and}\quad  $D^b(\mathcal{B}_{\Sigma}\text{-coh})$ respectively.

(2) In particular, the functors
$$D^b(\mathcal{A}_{\Sigma}\text{-coh})\to
D^b(\mathcal{A}_{\Sigma}\text{-mod})\quad \text{and}\quad
D^b(\mathcal{B}_{\Sigma}\text{-coh})\to D^b(\mathcal{B}_{\Sigma}\text{-mod})$$
are also full and faithful.
\end{thm}

\section{Koszul duality and the Serre functor}\label{kos dual and serre funct}

\subsection{} %The functor $\phi :D^b(\text{co-}\mathcal{A}_{\Sigma }\text{-tcf})\to D^b(Qcoh _{X,T})$.

Let us define an exact functor $\phi : \text{co-}\mathcal{A}_{\Sigma }\text{-Mod}\to C(Qcoh_{X,T})$ - the category of complexes in
$Qcoh_{X,T}$.
Fix an orientation of each simplex $\sigma \in \Sigma$. Let $\mathcal{N}=\{\mathcal{N}_\sigma \}
\in \text{co-}\mathcal{A}_{\Sigma }\text{-Mod}$. Then each $\mathcal{N}_\sigma$ is a graded $A(st({\bf o}_\sigma))$-module,
which corresponds to a $T$-equivariant quasi-coherent sheaf on
the affine open subset $st({\bf o}_\sigma)$. Let $\phi (\mathcal{N}_\sigma)\in Qcoh_{X,T}$ be the direct image of this sheaf under the inclusion $st({\bf o}_\sigma)\hookrightarrow X$. If $\tau \subset \sigma$ then the structure morphism of $A(st({\bf o}_\sigma))$-modules $\beta _{\tau \sigma}:\mathcal{N}_\tau \to \mathcal{N}_\sigma$ gives the morphism of equivariant quasi-coherent sheaves $\phi (\beta _{\tau \sigma}):\phi (\mathcal{N}_\tau) \to \phi (\mathcal{N}_\sigma)$. Define $\phi (\mathcal{N})$ as the complex
$$\phi (\mathcal{N}):=\cdots \bigoplus _{\dim (\tau )=i}\phi(\mathcal{N}_\tau) \stackrel{d^i}{\longrightarrow }
\bigoplus _{\dim (\sigma )=i+1}\phi(\mathcal{N}_\sigma) \cdots$$
where the summand $\phi(\mathcal{N}_\tau)$ appears in cohomological degree $\dim (\tau)$ and each differential $d^i$ is the sum of maps $\phi (\beta _{\tau \sigma})$ with $\pm$ sign depending on whether the orientations of $\tau$ and $\sigma$ agree or not. This defines the exact functor
$\phi :\text{co-}\mathcal{A}_{\Sigma }\text{-Mod}\to C(Qcoh_{X,T})$.
We will be interested in the restriction of this functor to the subcategory $\text{co-}\mathcal{A}_{\Sigma }\text{-tcf}\subset \text{co-}\mathcal{A}_{\Sigma }\text{-Mod}$ and the corresponding derived functor
$$\phi :D^b(\text{co-}\mathcal{A}_{\Sigma }\text{-tcf})\to D^b(Qcoh _{X,T})$$

Let $\psi :D^b(coh _{X,T})\to D^b(\mathcal{A}_{\Sigma }\text{-mod})$ be the fully faithful embedding (Theorem \ref{summary on der equiv and emb}) and denote the Koszul duality functor ${\bf K}_A$ in Theorem \ref{koz dual thm} by ${\bf K}$.

\begin{theorem} \label{main theorem nice formulation} Assume that the toric variety is complete (and smooth). Then the composition of functors $\phi \cdot {\bf K}\cdot \psi$ has its essential image in the full subcategory $D^b(coh_{X,T})\subset D^b(Qcoh_{X,T})$, and is isomorphic to the Serre functor
$$S:D^b(coh_{X,T})\to D^b(coh_{X,T}),\quad S(\mathcal{F})=\omega_X[n]\otimes _{\mathcal{O}_X}\mathcal{F}$$ (where $\omega _X$ has the natural equivariant structure).
That is we have the commutative diagram of functors
\begin{equation}\label{comm diag of functors main thm}
\begin{CD}
D^b(coh _{X,T}) @>S>> D^b(coh _{X,T})\\
@VV\psi V  @AA\phi A \\
D^b(\mathcal{A}_{\Sigma}\text{-mod}) @>{\bf K} >> D^b(\text{co-}\mathcal{A}_{\Sigma}\text{-tcf})
\end{CD}
\end{equation}
\end{theorem}

\subsection{Proof of Theorem \ref{main theorem nice formulation}}
 Since $\psi (\mathcal{O}_{X})=\mathcal{A}_{\Sigma}$, by Proposition \ref{addition to comb in case of complete fan} we have
\begin{equation}\label{formula for complete}
{\bf K}\cdot \psi (\mathcal{O}_X)\simeq \mathcal{A}_{\Sigma}^\vee
\end{equation}

Consider the complex in $C(Qcoh_{X,T})$
\begin{equation}\label{phi of str}
\phi (\mathcal{A}_{\Sigma}^\vee):\ \phi(\mathcal{A}_{\Sigma ,0}^\vee) \stackrel{d^0}{\to} \bigoplus _{\dim (\tau )=1}\phi(\mathcal{A}_{\Sigma ,\tau}^\vee) \stackrel{d^1}{\longrightarrow }
\bigoplus _{\dim (\sigma )=2}\phi(\mathcal{A}_{\Sigma ,\sigma}^\vee) \cdots
\end{equation}
where $0\in \Sigma$ is the zero cone, $\mathcal{A}_{\Sigma ,0}^\vee  \simeq \mathcal{A}_{\Sigma ,0}=A(T)$, and $\phi (\mathcal{A}_{\Sigma ,0}^\vee)$ is the direct image $j_*\mathcal{O}_T$ under the open embedding $j:T\hookrightarrow X$. In particular there is a natural injective morphism of equivariant sheaves $\mathcal{O}_X\to
\phi(\mathcal{A}_{\Sigma ,0}^\vee)$.

Recall that the canonical line bundle on $X$ corresponds to the negative of the "boundary divisor", i.e.
$$\omega _X=\mathcal{O}_X(-\Sigma D_i)$$
where $D_i$'s are the distinct $T$-invariant divisors. We have the composition of inclusions $\omega _X\subset \mathcal{O}_X\subset \phi(\mathcal{A}_{\Sigma ,0}^\vee)$ which we denote by $\eta$. Let us consider the augmented diagram

\begin{equation}\label{augm diagr}
0\to \omega _X\stackrel{\eta}{\to} \phi(\mathcal{A}_{\Sigma ,0}^\vee) \stackrel{d^0}{\to} \bigoplus _{\dim (\tau )=1}\phi(\mathcal{A}_{\Sigma ,\tau}^\vee) \stackrel{d^1}{\longrightarrow }
\bigoplus _{\dim (\sigma )=2}\phi(\mathcal{A}_{\Sigma ,\sigma}^\vee) \cdots
\end{equation}

\begin{lemma} \label{acycl of augm diag} The diagram \eqref{augm diagr} is an acyclic complex.
\end{lemma}

\begin{proof} It suffices to prove the assertion locally on $X$. Since $X$ is complete, it is covered by open subsets which are the stars of $T$-fixed points. Choose one such open subset $U$. Then $U$ is isomorphic as a toric variety to the affine $n$-space, i.e. $A(U)=\bbC [x_1,...,x_n]$ with the standard $T$-action. The category $Qcoh_{U,T}$ is equivalent to the category of graded $A(U)$-modules. Under this equivalence the dualizing sheaf $\omega _U$ corresponds to the principal ideal $(x_1\cdots x_n)\subset A(U)$. Choose a coordinate plane $L:\{x_{1}=\cdots =x_{s}=0\}\subset U$. It is the closure of a $T$-orbit ${\bf o}_\tau$ for a cone $\tau \in \Sigma$ of dimension $s$. Then the graded $A(U)$-module corresponding to the sheaf $\phi(\mathcal{A}_{\Sigma ,\tau}^\vee)$ is $\bbC [x_1^{-1},...,x_s^{-1},x_{s+1}^{\pm},...,x_n^{\pm}]$. (Similarly for the other $T$-orbits in $U$). So the fact that the restriction of the diagram \eqref{augm diagr} to $U$ is an exact complex can be verified by tracing the appearances of a fixed monomial $x_i^{\pm}$ in the modules $\phi(\mathcal{A}_{\Sigma ,\tau}^\vee)$. For example, if $n=2$, this diagram is the exact complex
$$0\to (x_1x_2) \to \bbC[x_1^{\pm},x_2^{\pm}]\to \left(\bbC[x_1^{-1},x_2^{\pm}]\oplus \bbC[x_1^{\pm},x_2^{-1}]\right)\to \bbC[x_1^{-1},x_2^{-1}]\to 0$$
\end{proof}

\begin{corollary}\label{isom on str one obj} There exists a canonical isomorphism of objects in  $D^b(coh_{X,T})$: $$\theta :\omega _X[n]\to \phi \cdot {\bf K}\cdot \psi (\mathcal{O}_X)$$
\end{corollary}

\begin{proof} This follows immediately from formula \eqref{formula for complete} and Lemma \ref{acycl of augm diag}.
\end{proof}

\begin{lemma}\label{last lemma} For any $\mathcal{F}\in D^b(coh_{X,T})$ there exists a functorial isomorphism
$$(\phi \cdot {\bf K}\cdot \psi (\mathcal{O}_{X}))\stackrel{\bf L}{\otimes}\mathcal{F}\to \phi \cdot {\bf K}\cdot \psi (\mathcal{F})$$
\end{lemma}

\begin{proof} Let $\mathcal{F}\in coh_{X,T}$ be locally free.
Recall that ${\bf K}\cdot \psi (\mathcal{F})$ is a complex consisting of direct sums of cosheaves
$$\mathcal{A}_{[\sigma]}^\vee \otimes _{\mathcal{A}_{\Sigma ,\sigma}}\psi(\mathcal{F})_\sigma=\mathcal{A}_{[\sigma]}^\vee \otimes _{A(st({\bf o}_\sigma))}\mathcal{F}(st({\bf o}_\sigma))$$
The stalk at $\tau \subset \sigma$ of such a cosheaf is the following graded $A(st({\bf o}_\tau))$-module which we consider as a quasi-coherent sheaf on $st({\bf o}_\tau)$:
$$A(st({\bf o}_\tau))^\vee \otimes _{A(st({\bf o}_\sigma))}\mathcal{F}(st({\bf o}_\sigma))\in Qcoh_{st({\bf o}_\tau) ,T}$$
Finally the functor $\phi$ takes the direct image under the open embedding $j:st({\bf o}_\tau)\hookrightarrow X$:
$$j_*\left(A(st({\bf o}_\tau))^\vee \otimes _{A(st({\bf o}_\sigma))}\mathcal{F}(st({\bf o}_\sigma))\right)\in Qcoh_{X,T}$$
which is naturally isomorphic to
$$j_*\left(A(st({\bf o}_\tau))^\vee \otimes _{A(st({\bf o}_\sigma))}\mathcal{O}_X(st({\bf o}_\sigma))\right)\otimes \mathcal{F}$$
This proves the lemma.
\end{proof}

Theorem \ref{main theorem nice formulation} now follows. Indeed, by Corollary \ref{isom on str one obj} and Lemma \ref{last lemma} we have for $\mathcal{F}\in D^b(coh_{X,T})$ the functorial isomorphism
$$\omega _X[n]\otimes \mathcal{F}\stackrel{\theta \otimes id}{\longrightarrow}\phi \cdot {\bf K}\cdot \psi (\mathcal{O}_X)\stackrel{{\bf L}}{\otimes}\mathcal{F}\simeq
\phi \cdot {\bf K}\cdot \psi (\mathcal{F})$$

\subsection{A comment on Theorem \ref{main theorem nice formulation}}

Starting with a locally free sheaf $\mathcal{F} \in coh_{X,T}$ the composition of functors $\phi \cdot {\bf K}\cdot \psi$ constructs a complex $\phi \cdot {\bf K}\cdot \psi(\mathcal{F})$ which is a resolution of the shifted sheaf $\omega _X[n]\otimes _{\mathcal{O}_X}\mathcal{F}$. This is the $T$-equivariant Cousin resolution
$$\mathcal{\omega} _X[n]\otimes _{\mathcal{O}_X}\mathcal{F} \to Cousin(\omega _X[n]\otimes _{\mathcal{O}_X}\mathcal{F})$$ as defined in [Kempf] with respect to the filtration of $X$ by closed subsets
$$X=Z_0\supset Z_1\supset \cdots$$
where
$$Z_i:=\coprod _{\text{codim}({\bf o})\geq i}{\bf o}$$
(See Appendix).

\section{Appendix. Local cohomology and Cousin complex}

We recall the Cousin complex in the equivariant setting according to \cite{Kempf}. Let $X$ be a smooth $T$-toric variety.

Let $W\subset Z\subset X$ be closed subsets. We have the left exact functor of
"sections supported on $Z$"
$$\underline{\Gamma}_Z(-):Qcoh_X\to Qcoh_X$$
and its right derived functors $\mathcal{H}_Z^j$. Similarly for $W$. Define the functor $\underline{\Gamma}_{Z/W}(-)$ as the quotient
$$\underline{\Gamma}_{Z/W}(-):=\underline{\Gamma}_{Z}(-)/
\underline{\Gamma}_{W}(-)$$
This functor is not necessarily left exact, but still we consider its right derived functors
$$\mathcal{H}^j_Z(-):={\bf R}^j\underline{\Gamma}_{Z/W}(-):Qcoh_X\to Qcoh_X$$

Given closed subsets $Z_1\supseteq Z_2\supseteq Z_3$ in $X$ and $\mathcal{F}\in Qcoh_X$ one has the long exact sequence in $Qcoh_X$
\begin{equation}\label{long ex of loc coh}
0\to \mathcal{H}^0_{Z_2/Z_3}(\mathcal{F})\to
\mathcal{H}^0_{Z_1/Z_3}(\mathcal{F})\to \mathcal{H}^0_{Z_1/Z_2}(\mathcal{F})\stackrel{d}{\to }\mathcal{H}^1_{Z_2/Z_3}(\mathcal{F})\to ...
\end{equation}

Consider a decreasing filtration of the space $X$ by closed subsets
\begin{equation} \label{decr filtr}
X=Z_0\supseteq Z_1\supseteq ...
\end{equation}
Then for any $\mathcal{F}\in Qcoh_X$ the connecting homomorphisms $d$ in \eqref{long ex of loc coh} give rise to the (local) Cousin complex
\begin{equation}\label{loc cous compl}
Cousin(\mathcal{F}):\quad  \mathcal{H}^0_{Z_0/Z_1}(\mathcal{F})\stackrel{d_0}{\to}
\mathcal{H}^1_{Z_1/Z_2}(\mathcal{F})\stackrel{d_1}{\to}
\mathcal{H}^2_{Z_2/Z_3}(\mathcal{F})\stackrel{d_2}{\to}\cdots
\end{equation}
which comes with an augmentation map $\mathcal{F}\stackrel{e}{\to}Cousin(\mathcal{F})$.

Let us now choose the closed subsets as follows
$$Z_i:=\coprod _{\text{codim}({\bf o})\geq i}{\bf o}$$
The following theorem summarizes the properties of the complex $Cousin(\mathcal{F})$ that are relevant for us.

\begin{thm} Assume that $\mathcal{F}$ is a locally free coherent sheaf on $X$.

(1) The augmented Cousin complex $\mathcal{F}\stackrel{e}{\to}Cousin(\mathcal{F})$ is exact.

(2) Let $i_{{\bf o}}:{\bf o}\hookrightarrow X$ be the locally closed inclusion of the orbit ${\bf o}$. Then
\begin{equation}\label{dir sum over orbits}
\mathcal{H}^i_{Z_i/Z_{i+1}}(\mathcal{F})=\bigoplus _{\text{codim}({\bf o})=i}i_{{\bf o}*}\mathcal{H}^i_{{\bf o}}(\mathcal{F}\vert _{st({\bf o})})
\end{equation}
and $\mathcal{H}^j_{Z_i/Z_{i+1}}(\mathcal{F})=0$ for $j\neq i$.

%(3) Let the orbits ${\bf o}$ and ${\bf o}'$ have codimension $i$ and $i+1$ respectively and assume that ${\bf o}'\subset \overline{{\bf o}}$. Then there exists a canonical morphism
%\begin{equation}\label{can delta}
%\delta _{{\bf o}{\bf o}'}: i_{{\bf o}*}\mathcal{H}^i_{{\bf o}}(\mathcal{F}\vert _{st({\bf o})})\to
%i_{{\bf o}'*}\mathcal{H}^{i+1}_{{\bf o}'}(\mathcal{F}\vert _{st({\bf o}')})
%\end{equation}
%such that the differential $d_i:\mathcal{H}^i_{Z_i/Z_{i+1}}(\mathcal{F})\to \mathcal{H}^{i+1}_{Z_{i+1}/Z_{i+2}}(\mathcal{F})$ in the Cousin complex is the sum of morphisms $\pm \delta _{{\bf o}{\bf o}'}$ under the decomposition \eqref{dir sum over orbits}.

Now assume in addition that $\mathcal{F}\in coh_{X,T}$.

(3) The augmented Cousin complex $\mathcal{F}\stackrel{e}{\to}Cousin(\mathcal{F})$ is a complex in $Qcoh_{X,T}$. Moreover, the sheaves $\mathcal{H}^i_{Z_i/Z_{i+1}}(\mathcal{F})\in Qcoh_{X,T}$ are injective, hence $Cousin(\mathcal{F})$ is an injective resolution of $\mathcal{F}$.
\end{thm}

\begin{proof} (1) The following assumptions are satisfied: (a) The sheaf $\mathcal{F}$ is a Cohen-Macauley $\mathcal{O}_X$-module, (b) $\text{codim}(Z_i)\geq i$ for all $i$, (c) the variety
$$Z_i-Z_{i+1}=\coprod _{\text{codim}({\bf o})=i}{\bf o}$$
is affine for all $i$. Hence the augmented Cousin complex $\mathcal{F}\stackrel{e}{\to}Cousin(\mathcal{F})$ is exact by Theorem 10.9 in [Kempf].

(2) Let $U_i:=X-Z_i$. Since $\mathcal{F}$ is locally free, the local cohomology sheaves $\mathcal{H}^j_{Z_i-Z_{i+1}}(\mathcal{F}\vert _{U_i})$ are zero if $j\neq i$ and $\mathcal{H}^i_{Z_i-Z_{i+1}}(\mathcal{F}\vert _{U_i})\in Qcoh_{U_i}$ is supported on $Z_i-Z_{i+1}$. Moreover
$$\mathcal{H}^i_{Z_i-Z_{i+1}}(\mathcal{F}\vert _{U_i})=\bigoplus _{\text{codim}{\bf o}=i}\mathcal{H}^i_{{\bf o}}(\mathcal{F}\vert _{st({\bf o})})$$

The condition (L.V.) on p. 362 of \cite{Kempf} is satisfied. Hence by Lemma 8.5 (e) in loc. cit. the equality \eqref{dir sum over orbits} holds.

(3) This follows from Theorem 11.6(e) in \cite{Kempf}.
\end{proof}


\begin{thebibliography}{9999999}

\bibitem[AM]{AM} M. F. Atiyah, I. G. MacDonald, Introduction to commutative algebra, 1969.

\bibitem[BeLu]{BeLu} J.\,Bernstein, V.\,A.\,Lunts, Equivariant sheaves and functors, LNM 1578.

\bibitem[BrLu]{BrLu} T.\,Braden, V.\,A.\,Lunts, Equivariant constructible Koszul duality for dual toric varieties,
    Adv. Math. 201 (2006), no. 2, 408–453

\bibitem[Danilov]{Danilov} V.\,I.\,Danilov, The geometry of toric varieties, Russian Mathematical Surveys, 1978, Volume 33, Issue 2, pages 97-154.

\bibitem[Fulton]{Fulton} W.\,Fulton, Introduction to toric varieties,
Princeton University Press, 1993.

\bibitem[Godement]{Godement} R.\,Godement, Topologie algébrique et théorie des faisceaux, Hermann, Paris, 1958, viii+283 pp

\bibitem[Kempf]{Kempf} G.\,Kempf, The Grothendieck-Cousin complex of an induced representation, Advances of Mathematics 29, 310-396 (1978).

\bibitem[Lu]{Lu} V.\,A.\,Lunts, Equivariant sheaves on toric varieties, Compositio Math. 96 (1995), no.1, 63-83.










\end{thebibliography}
\end{document}